\documentclass{eptcs-modified}
\usepackage{marginnote}
\usepackage{amsmath}
\usepackage{amssymb}
\usepackage{url}
\usepackage{amsthm}
\usepackage{graphicx}
\usepackage{amscd}
\usepackage[usenames,dvipsnames,svgnames,table]{xcolor}
\usepackage{mathrsfs}
\usepackage{color}
\usepackage{enumitem}
\usepackage{comment}

\theoremstyle{definition}
\newtheorem{theorem}{Theorem}
\newtheorem{definition}[theorem]{Definition}

\newtheorem{corollary}[theorem]{Corollary}
\newtheorem{proposition}[theorem]{Proposition}
\newtheorem{example}[theorem]{Example}
\newtheorem{lemma}[theorem]{Lemma}
\newtheorem{observation}[theorem]{Observation}
\newtheorem{fact}[theorem]{Fact}
\newtheorem{question}[theorem]{Open Question}

\newcommand{\uint}{{[0, 1]}}
\newcommand{\Cantor}{{\{0,1\}^\mathbb{N}}}

\newcommand{\Baire}{\mathbb{N}^\mathbb{N}}
\newcommand{\hide}[1]{}

\newcommand{\Varf}{\text{Var}_f}
\newcommand{\KO}{\ensuremath{\mathcal{O}}}
\newcommand{\CK}{\omega_1^\mathrm{CK}}
\newcommand{\cdf}{\operatorname{cdf}}

\newcommand{\uhr}{\upharpoonright}
\newcommand{\concentrate}{\operatorname{Concentrate}}

\definecolor{greenish}{rgb}{0.0,0.6,0.2}
\definecolor{lightgreenish}{rgb}{0.0,0.8,0.2667}
\definecolor{lightblueish}{rgb}{0.85,0.95,1}

\usepackage{soul}

\begin{document}

\title{Luzin's (N) and randomness reflection}

\author{
Arno Pauly
\institute{Department of Computer Science\\Swansea University\\Swansea, UK\\}
\email{Arno.M.Pauly@gmail.com}
\and
Linda Westrick
\institute{Department of Mathematics\\Penn State University\\University Park, PA, US\\}
\email{westrick@psu.edu}
\and
Liang Yu
\institute{Mathematical Department\\Nanjing University\\Nanjing City, China\\}
\email{yuliang.nju@gmail.com}
}

\def\titlerunning{Randomness reflection}
\def\authorrunning{A.~Pauly, L.~Westrick \& L.~Yu }
\maketitle

\begin{abstract}
{We show that a computable function $f:\mathbb R\rightarrow\mathbb R$ has Luzin's property (N) if
and only if it reflects $\Pi^1_1$-randomness, if
and only if it reflects $\Delta^1_1(\KO)$-randomness,
and if and only if it reflects $\KO$-Kurtz randomness,
but reflecting Martin-L\"of randomness or
weak-2-randomness does not suffice.
Here a function $f$ is said to reflect a randomness notion
$R$ if whenever $f(x)$ is $R$-random, then $x$
is $R$-random as well.  If additionally $f$ is known to have
bounded variation, then we show $f$ has Luzin's (N)
if and only if it reflects weak-2-randomness, and if
and only if it reflects $\emptyset'$-Kurtz randomness.}
This links classical real analysis with algorithmic randomness.
\end{abstract}

\section{Introduction}
We revisit a notion from classic real analysis, namely Luzin's property (N), from the perspective of computability theory. A function $f : \mathbb{R} \to \mathbb{R}$ has Luzin's (N), if the image of any (Lebesgue) null set under $f$ has again measure $0$.
{This concept was studied extensively by Luzin in his thesis}
\cite{Luzin1951}.
For functions with bounded variation, this notion is just equivalent to absolutely continuous functions -- but already for general continuous functions, Luzin's (N) is a somewhat intricate property. A formal result amounting to this was obtained by Holick\'y, Ponomarev, Zaj\'j\v{c}ek and Zelen\'y, showing that the set of functions with Luzin's (N) is $\Pi^1_1$-complete in the space of continuous functions \cite{holicky}.

From a computability-theoretic perspective, Luzin's (N) is readily seen to be some kind of randomness reflection: By contraposition, it states that whenever $f[A]$ has positive measure (i.e.~contains a random point for a suitable notion of randomness), then $A$ has positive measure, too (i.e.~contains a random point). It thus seems plausible that for some suitable randomness notion, Luzin's (N) for computable functions is equivalent to saying that whenever $f(x)$ is random, then so is $x$. Our main finding (Theorem \ref{theo:main}) is that this is indeed the case, and that $\Pi^1_1$-randomness is such a suitable randomness notion. An indication that this is a non-trivial result is that our proof uses ingredients such as Friedman's conjecture (turned into a theorem by Martin \cite{friedman3,martin4,yuliang3}).

While the exploration of how randomness interacts with function application, and the general links to real analysis, has a long tradition (see e.g.~the survey by Rute \cite{rute}), the concepts of randomness preservation (if $x$ is random, so is $f(x)$) and no-randomness-from-nothing (if $y$ is random, then there is some random $x \in f^{-1}(y)$) have received far more attention than randomness reflection. Our results not only fill this gap, but may shed a light on why randomness reflection has been less popular: As the most natural notion of randomness reflection turns out to be $\Pi^1_1$-randomness reflection, we see that studying higher randomness is essential for this endeavour.

Our theorems and proofs generally refer to computability. However, we stress that since the results relativize, one can obtain immediate consequences in classic real analysis. An example of this is Corollary \ref{corr:banach}, which recovers a theorem by Banach. More such examples  can be found in Section {8}, where, by applying  relativized computability method,  we are able to prove some results in classical analysis. While we are not aware of such consequences that would advance the state of the art in real analysis, it is {plausible} that future use of our techniques could accomplish this.

\paragraph*{Overview of our paper}
In Section \ref{sec:higherrandomness} we do not discuss randomness reflection at all, but rather prove a result in higher randomness of independent interest. Theorem \ref{theo:milleryu} is of the form ``if a somewhat random $X$ is hyp-computed by a very random $Y$, then $X$ is already very random". It is the higher randomness analog of \cite[Theorem 4.3]{yuliang} by Miller and the third author. This result is a core ingredient of our main theorem.

Section \ref{sec:mlr} contains the main theorem of our paper, the equivalence of Luzin's (N) for computable functions with $\Pi^1_1$-randomness reflection and with $\Delta^1_1(\KO)$-randomness reflection. We consider higher Kurtz randomness in Section \ref{sec:kurtz}, and show that for continuous functions $f:\mathbb R\rightarrow \mathbb R$, Luzin's (N) is equivalent to the reflection of $\KO$-Kurtz randomness, and separate this from $\Delta^1_1$-Kurtz randomness reflection.
In Subsection \ref{subsec:questions} we discuss the open questions raised by our main theorem: Just because Luzin's (N) is equivalent to reflection of several higher randomness notions does not mean that it cannot be also equivalent to randomness reflection for some ``lower'' notions. For Martin-L\"of-randomness reflection and weak-2-randomness reflection, we provide a separation from Luzin's (N) in Section \ref{sec:separating}.

In Sections \ref{sec:boundedvariation} and \ref{sec:measures} we consider Luzin's (N) for more restricted classes of functions, namely functions with bounded variation and strictly increasing functions. Here Luzin's (N) turns out to be equivalent to weak-2-randomness reflection, but we can still separate it from several other randomness-reflection-notion. These investigations tie in to a project by Bienvenu and Merkle \cite{merkle3} regarding how two computable measures being mutually absolutely continuous (i.e.~having the same null sets) relates to randomness notions for these measures coinciding.

In Section \ref{sec:generichardness} we take a very generic look at the complexity of randomness reflection, and show that the $\Pi^1_1$-hardness established for Luzin's (N) in \cite{holicky} applies to almost all other randomness reflection notions, too.

Section \ref{sec:digression} contains a brief digression about functions where the image of null sets is small in some other sense (countable or meagre). We prove these  classical analysis results via various classical and higher computability methods.

We then conclude in Section \ref{sec:outlook} with a discussion of how this line of investigation could be continued in the future.

\section{Randomness and hyperarithmetic reductions}
\label{sec:higherrandomness}

Throughout, we assume familiarity with the theory of algorithmic
randomness and higher randomness in particular.  A standard
references for the former are \cite{downeyhirschfeldt_ARC} {and }\cite{Niesbook09}. {For the latter, readers may refer to }\cite{CY15book}.
We use standard computability-theoretic notation.  The Lebesgue
measure is denoted by $\lambda$.

Our goal in this section is to establish the following:

\begin{theorem}
\label{theo:milleryu}
Let $Y$ be $\Delta^1_1(Z)$-random and $\Pi^1_1$-random, let $X$ be $\Delta^1_1$-random and let $X \leq_h Y$. Then $X$ is $\Delta^1_1(Z)$-random.
\end{theorem}

This is a higher-randomness counterpart to \cite[Theorem 4.3]{yuliang}, and the proof proceeds by adapting both this and \cite[Lemma 4.2]{yuliang}. We will use the theorem in the following form:

\begin{corollary}
Let $Z \geq \KO$. If $X$ is $\Delta^1_1$-random, $Y$ is $\Delta^1_1(Z)$-random and $X \leq_h Y$, then $X$ is $\Delta^1_1(Z)$-random.
\end{corollary}

\begin{lemma}
\label{lemma:42}
Fix $\alpha \in \KO$ and $e \in \mathbb{N}$. If $X$ is $\Delta^1_1$-random, then:
$$\exists c \ \forall n \ \lambda(\{Y \mid \varphi_e(Y^{(\alpha)}) \in [X_{\upharpoonright n}]\}) < 2^{-n+c}$$
\begin{proof}
Analogous to the proof of \cite[Lemma 4.2]{yuliang}. Let $\mathcal{H}_\sigma = \{Y \mid \varphi_e(Y^{(\alpha)}) \in [\sigma]\}$, and then let $F_i = \{\sigma \mid \lambda(\mathcal{H}_\sigma) > 2^{-|\sigma| + i}\}$. By construction, the $\mathcal{H}_{\sigma}$ are uniformly $\Delta^0_{\alpha+1}$ (as subsets of $\Cantor$), and so the sets $F_i$ are uniformly $\Delta^0_{\alpha+2}$ (as subsets of $2^{<\omega}$).

A counting argument shows that $\lambda([F_i]) < 2^{-i}$: Pick a prefix free set $D \subseteq F_i$ with $[D] = [F_i]$. Then: $$1 \geq \lambda(\bigcup_{\sigma \in D} \mathcal{H}_\sigma) = \sum_{\sigma \in D} \lambda(\mathcal{H}_\sigma) \geq \sum_{\sigma \in D} 2^{-|\sigma| + i} = 2^{i}\lambda([F_i])$$

We see that $([F_i])_{i \in \mathbb{N}}$ is a Martin-L\"of test relative to $\emptyset^{\alpha+2}$. Since $X$ is $\Delta^1_1$-random, there has to be some $c \in \mathbb{N}$ with $X \notin [F_c]$. This in turn means that $\forall n \in \mathbb{N} \ X_{\upharpoonright n} \notin F_c$, which by definition of $F_c$ is the desired claim.
\end{proof}
\end{lemma}

\begin{fact}[Sacks \cite{Sacks69}]
\label{fact:twodelta11randomness}
{$\Delta^1_1(Z)$-randomness (defined by being contained in no $\Delta^1_1(Z)$-null sets) is equivalent to being $\hat{Z}$-random for every $\hat{Z} \in \Delta^1_1(Z)$.}
\end{fact}

\begin{lemma}
\label{lemma:milleryu}
Fix $\alpha \in \KO$ and $e \in \mathbb{N}$. If $X = \varphi_e(Y^{(\alpha)})$, $X$ is $\Delta^1_1$-random and $Y$ is $\Delta^1_1(Z)$-random, then $X$ is {$\Delta^1_1(Z)$}-random.
\begin{proof}
We follow the proof of \cite[Theorem 4.3]{yuliang}. Let $c$ be the constant guaranteed for $X$ by Lemma \ref{lemma:42}. As in the proof of Lemma \ref{lemma:42}, let $\mathcal{H}_\sigma = \{W \mid \varphi_e(W^{(\alpha)}) \in [\sigma]\}$. Let $\mathcal{G}_\sigma = \mathcal{H}_\sigma$ if $\lambda(\mathcal{H}_\sigma) < 2^{-|\sigma| + c}$ and $\mathcal{G}_\sigma = \emptyset$ else. Note that $\mathcal{G}_\sigma$ is still uniformly $\Delta^1_1$. The choice of $c$ in particular guarantees that $Y \in \mathcal{G}_{X_{\upharpoonright n}}$ for each $n \in \mathbb{N}$.

Let {$\cap_n U_n$ denote a Martin-L\"of test relative to $\hat Z$ for some
$\hat Z \in \Delta^1_1(Z)$.}  By Fact \ref{fact:twodelta11randomness},
it suffices to show that $X\not\in \cap_n U_n$.
We set {$\mathcal{K}_i = \bigcup_{\sigma \in U_{c+i}} \mathcal{G}_\sigma$}, and $\mathcal{K} = \bigcap_{i \in \mathbb{N}} \mathcal{K}_i$. We find that $\mathcal{K}$ is $\Delta^1_1(Z)$. Moreover, we have that:
$$\lambda(\mathcal{K}_i) \leq \sum_{\sigma \in U_{c+i}} \lambda(\mathcal{G}_\sigma) \leq \sum_{\sigma \in U_{c+i}} 2^{-|\sigma| + c} \leq 2^{-i}$$
Hence, it follows that $\lambda(\mathcal{K}) = 0$, {so for some $i$, $Y \not \in \mathcal K_i$.
Then $X \not \in U_{c+i}$, because $Y \in \mathcal G_\sigma$ for all $\sigma \prec X$.}
\end{proof}
\end{lemma}

\begin{fact}[Sacks \cite{Sacks69}]
\label{fact:omega1ck}
If $Y$ is $\Pi^1_1$-random, then $\omega_1^{\mathrm{CK}} = \omega_1^{\mathrm{CK},Y}$.
\end{fact}

\begin{proof}[Proof of Theorem \ref{theo:milleryu}]
Since $Y$ is $\Pi^1_1$-random, we know that $X \leq_h Y$ implies the existence of some $\alpha \in \KO$ and $e \in \mathbb{N}$ such that $X = \varphi_e(Y^{(\alpha)})$ (rather than merely $\alpha \in \KO^Y$). We can thus invoke Lemma \ref{lemma:milleryu} {to conclude that $X$ is $\Delta^1_1(Z)$-random.}
\end{proof}

{As an aside,}
the requirement in Theorem \ref{theo:milleryu} that $Y$ be $\Pi^1_1$-random might be unexpected at first -- it has no clear counterpart in \cite[Theorem 4.3]{yuliang}.
{The following example,
which is not needed for anything else in the paper, shows that this assumption is necessary.}

\begin{example}
There are $\Delta^1_1$-random $X$ and $\Delta^1_1(Z)$-random $Y$ with $X \leq_h Y$ but $X$ is not $\Delta^1_1(Z)$-random. In fact, we shall chose $X = Z$, and make $X$ even $\Pi^1_1$-random.
\begin{proof}
Let $Y$ be a $\Delta^1_1$-random satisfying $Y \geq_h \KO$. The existence of such a $Y$ was shown in \cite{yuliang2}. Let $X$ be Martin-L\"of random relative to $Y \oplus \KO$ while satisfying $X \leq_h Y$. This choice ensures that $X$ is $\Pi^1_1$-random (so in particular $\Delta^1_1$-random).

By van Lambalgen's theorem {relativized to $\emptyset^{(\alpha)}$},
if both $X$ and $Y$ are $\Delta^1_1$-random, then for any $\alpha < \omega_1^{\mathrm{CK}}$ it holds that $X$ is $Y \oplus \emptyset^{(\alpha)}$-random iff $X \oplus Y$ is $\emptyset^{(\alpha)}$-random iff $Y$ is $X \oplus \emptyset^{(\alpha)}$-random. Since by choice of $X$, we know that in particular $X$ is $Y \oplus \emptyset^{(\alpha)}$-random, we conclude that $Y$ is $X \oplus \emptyset^{(\alpha)}$-random.

From (\cite[Corollary 4.3]{CY15}) it follows that for $\Pi^1_1$-random $X$ and $\beta < \omega_1^{\mathrm{CK}}$ it holds that $X^{(\beta)} \leq_{\mathrm{T}} X \oplus \emptyset^{(\beta)}$.
{(The conclusion above surely does not require full
$\Pi^1_1$-randomness of $X$, but too much precision would take us afield.)}
Together with the above, this shows that $Y$ is $X^{(\alpha)}$-random for every $\alpha < \omega_1^{\mathrm{CK}}$. Since $X$ is $\Pi^1_1$-random, by Fact \ref{fact:omega1ck} we have that $\omega_1^{\mathrm{CK}} = \omega_1^{\mathrm{CK},X}$, and thus that $Y$ is $Z$-random for any $Z \in \Delta^1_1(X)$. By Fact \ref{fact:twodelta11randomness} this establishes $Y$ to be $\Delta^1_1(X)$-random. But trivially, $X$ cannot be $\Delta^1_1(X)$-random.
\end{proof}
\end{example}

\section{Luzin's (N) and randomness reflection}
\label{sec:main}

\begin{definition}
A function satisfies Luzin's (N) iff the image of every null set is null.
\end{definition}

\begin{definition}
{For any randomness notion $R$ and a function $f$, we say that $f$ \emph{reflects $R$-randomness}
if $f(x)$ is $R$-random implies $x$ is $R$-random for all $x$ in the domain of $f$.}
\end{definition}

\subsection{Luzin's (N) and higher Martin-L\"of randomness reflection}\label{sec:mlr}

By noting that the sets of points not Martin-L\"of random relative to some oracle are canonical choices of null sets, we obtain the following:

\begin{proposition}
\label{prop:variants}
The following are equivalent for a computable function $f : \mathbb{R} \to \mathbb{R}$:
\begin{enumerate}
\item $f$ satisfies Luzin's (N)
\item $\forall p \in \Cantor \ \exists q \in \Cantor \ f(x) \in \mathrm{MLR}(q) \Rightarrow x \in \mathrm{MLR}(p)$
\item $\forall p \in \Cantor \ f(x) \in \Delta^1_1\mathrm{-random}(p) \Rightarrow x \in \mathrm{MLR}(p)$
\item $\forall p \in \Cantor \ f(x) \in \Delta^1_1\mathrm{-random}(p) \Rightarrow x \in \Delta^1_1\mathrm{-random}(p)$
\end{enumerate}
\begin{proof}
\begin{description}
\item[$1. \Leftrightarrow 2.$]
Each null set is contained in a set of the form $\mathrm{MLR}(q)^C$ for some oracle $q \in \Cantor$. Luzin's (N) is thus equivalent to saying that for any $p$ there is a $q$ with $f[\mathrm{MLR}(p)^C] \subseteq \mathrm{MLR}(q)^C$. Taking the contrapositive yields $(2)$.
\item[$2. \Rightarrow 3.$] Any $\Sigma^1_1$-null set is contained in a $\Delta^1_1$-null set (\cite{Sacks69}). Thus it suffices to choose $q$ as something hyperarithmetical in $p$.
\item[$3. \Rightarrow 1.$] Trivial.
\item[$3. \Rightarrow 4.$] Assume that $(4)$ fails, i.e.~that there is some $p \in \Cantor$ and some $x \notin \Delta^1_1\mathrm{-random}(p)$ with $f(x) \in \Delta^1_1\mathrm{-random}(p)$. But if $x \notin \Delta^1_1\mathrm{-random}(p)$, then there is some $q \leq_h p$ with $x \notin \mathrm{MLR}(q)$, but $f(x) \in \Delta^1_1\mathrm{-random}(q) = \Delta^1_1\mathrm{-random}(p)$, hence $(3)$ is violated, too.
\item[$4. \Rightarrow 3.$] Trivial.
\end{description}
\end{proof}
\end{proposition}

\begin{corollary}
\label{corr:delta11reflectionfollows}
A computable function satisfying Luzin's (N) reflects $\Delta^1_1$-randomness
{relative to any oracle}.
\end{corollary}
We can now ask whether reflecting $\Delta^1_1$-randomness relative to some specific oracle already suffices.

\begin{fact}[\cite{martin4,yuliang3}]
\label{fact:pi01classko}
If $A$ is an uncountable $\Delta^1_1(y)$-class such that $y\leq_h z$ for every $z\in A$, then there is some $x \in A$ with $\KO^y \leq_h x$.
\end{fact}

\begin{fact}[Sacks \cite{Sacks69}]
\label{fact:delta11randomnonKO}
If $\KO \leq_h x$, then $x$ is not $\Delta^1_1(\KO)$-random.
\end{fact}

\begin{corollary}
\label{cor:delta11korandom}
If computable $f$ reflects $\Delta^1_1(\KO)$-randomness, then  for any $\Delta^1_1(\KO)$-random $y$ we find that $f^{-1}(y)$ is countable.
\end{corollary}
\begin{proof}
Assume that $y$ is $\Delta^1_1(\KO)$-random and $f^{-1}(y)$ is uncountable. Then by Fact \ref{fact:pi01classko}, there is some $x \in f^{-1}(y)$ with $\KO \leq_h x$. By Fact \ref{fact:delta11randomnonKO}, we find that $x$ is not $\Delta^1_1(\KO)$-random, contradicting that $f$ reflects $\Delta^1_1(\KO)$-randomness.
\end{proof}

\begin{observation}
\label{obs:countablefibers}
The following are equivalent for computable $f : \mathbb{R} \to \mathbb{R}$:
\begin{enumerate}
\item For almost all $y$ it holds that $f^{-1}(\{y\})$ is countable.
\item For every $\Delta^1_1$-random $y$ it holds that $f^{-1}(\{y\})$ is countable.
\end{enumerate}
\begin{proof}
The implication $2 \Rightarrow 1$ is trivial. For the other direction, note that $$\{y \mid f^{-1}(\{y\}) \textnormal{ is uncountable}\}$$ is $\Sigma^1_1$. This holds because for a $ \Sigma^1_1$-set $A$, being uncountable is equivalent to containing an element which is not hyperarithmetic relative to $A$. Due to Kleene's HYP-quantification theorem, an existential quantifier over non-hyperarithmetic elements is equivalent to an unrestricted existential quantifier. By assumption, it is a null set.
Any $\Sigma^1_1$-null set is contained in a $\Delta^1_1$-null set, so it is then contained in a $\Delta^1_1$-null set, and so cannot contain any $\Delta^1_1$-randoms.
\end{proof}
\end{observation}

\begin{theorem}
\label{theo:main}
The following are equivalent for computable $f : \mathbb{R} \to \mathbb{R}$:
\begin{enumerate}
\item $f$ satisfies Luzin's (N)
\item $f$ reflects $\Delta^1_1(\KO)$-randomness.
\item $f$ reflects $\Delta^1_1$-randomness and for almost all $y$, $f^{-1}(y)$ is countable.
\item $f$ reflects $\Pi^1_1$-randomness.
\end{enumerate}
\begin{proof}
That $(1)$ implies $(2)$ follows from Proposition \ref{prop:variants}. To see that $(2)$ implies $(1)$, we show that if $f$ reflects $\Delta^1_1(\KO)$-randomness, then it reflects $\Delta^1_1(r)$-randomness for all $r \geq_T \KO$.  {This will be enough because if $f$
reflects $\Delta^1_1(r)$-randomness for all $r \geq_T \KO$, then
for any $p \in \{0,1\}^\mathbb N$,
if $f(x) \in \mathrm{MLR}(\KO^{p\oplus \KO})$, then $f(x)$ is $\Delta^1_1(p\oplus \KO)$-random,
so $x$ is $\Delta^1_1(p \oplus \KO)$-random, so $x \in \mathrm{MLR}(p)$.
Thus $f$ satisfies
condition (2) of Proposition} \ref{prop:variants}.

So let $y$ be $\Delta^1_1(r)$-random for some $r \geq_T \KO$. By Corollary \ref{cor:delta11korandom}, $f^{-1}(y)$ is countable. So if $x \in f^{-1}(y)$, then $x \leq_h y$. Since $f$ reflects $\Delta^1_1(\KO)$-randomness, we know that $x$ is $\Delta^1_1$-random. We can thus invoke Theorem \ref{theo:milleryu} to conclude that $x$ is $\Delta^1_1(r)$-random, and have reached our goal.  (Note that $y$ is $\Pi^1_1$-random because $r\geq_T \KO$.)

To see that $(1)$ implies $(3)$ we use Proposition \ref{prop:variants} and Corollary \ref{cor:delta11korandom}. The proof that $(3)$ implies $(1)$ proceeds analogously to the proof that $(2)$ implies $(1)$, except that we conclude that $f^{-1}(\{y\})$ is countable from Observation \ref{obs:countablefibers} rather than Corollary \ref{cor:delta11korandom}.

To see that $(3)$ implies $(4)$, by Observation \ref{obs:countablefibers}, in fact the inverse image of every
$\Delta^1_1$-random point is countable.  In particular if $y$ is $\Pi^1_1$-random, then $f^{-1}(y)$ is countable.
We use the following characterization of $\Pi^1_1$-randomness due to Stern:
$y$  is $\Pi^1_1$-random if and only if $y$  is  $\Delta^1_1$-random and $\omega_1^y = \CK$  \cite{Stern1973, Stern1975, yuliang2}.
Recall also that $\omega_1^y =  \CK$ if and only if  $y \not\geq_h  \KO$ (\cite[Corollary 7.7]{sacks-hrt}).
Let $x \in f^{-1}(y)$.
By $\Delta^1_1$-randomness reflection, $x$ is $\Delta^1_1$-random.   Since  $f^{-1}(y)$ is countable,  $y \geq_h x$.
Thus  $x \not\geq_h \KO$.  Therefore  $x$ is $\Pi^1_1$-random.

The proof that $(4)$ implies $(1)$ is the same as how conditions $(2)$ or $(3)$ implied $(1)$.
Letting $y$ be $\Delta^1_1(r)$-random,
since $r\geq_T \KO$ then $y$ is $\Pi^1_1$-random.  If
$f^{-1}(y)$ were uncountable, by Fact \ref{fact:pi01classko}, it would contain some $x$ with $x \geq_h \KO$,
contradicting $\Pi^1_1$-randomness reflection.
\end{proof}
\end{theorem}

\subsection{A note on the countability of fibers}\label{sec:fibers}

We obtain as a corollary a reproof of a theorem by Banach \cite{banach2}, cf.~\cite[Chapter IX, Theorem 7.3]{saks}:

\begin{corollary}
\label{corr:banach}
If $f$ is continuous and satisfies Luzin's (N), then for almost all $y$ we find that $f^{-1}(y)$ is countable.
\end{corollary}

{The following generalization to measurable functions was also known, but we give
a new proof.}

\begin{corollary}
\label{cor:delta11korandom2}
\begin{enumerate}
\item If $f$ satisfies Luzin's (N) and there are a continuous function $g$ and a Borel set $A$ so that $f$ agrees with $g$ on $A$, then for almost every real $y \in f(A)$ we find that $f^{-1}(y)\cap A$ is countable.
\item If $f$ satisfies Luzin's (N) and is measurable, then for almost every real $y$ we find that $f^{-1}(y)$ is countable.
\end{enumerate}
\begin{proof}
(1). Fix a real $x$ so that $g$ is computable in $x$ and $A$ is $\Delta^1_1(x)$. Assume that $y$ is $\Delta^1_1(\KO^x)$-random and $f^{-1}(y)\cap A=g^{-1}(y)\cap A$ is uncountable. Since $ A$ is $\Delta^1_1(x)$,  by Claim \ref{fact:pi01classko}, then there is some $z \in g^{-1}(y)\cap A=f^{-1}(y)\cap A $ with $\KO^x \leq_h z\oplus x$. By Fact \ref{fact:delta11randomnonKO}, we find that $z$ is not $\Delta^1_1(\KO^x)$-random, contradicting that $g$ reflects $\Delta^1_1(\KO^x)$-randomness.

(2). By Luzin's theorem, there are a sequence Borel sets $\{A_n\}_{n\in \omega}$ and continuous functions $\{g_n\}_{n\in \omega}$ so that $\mathbb{R}\setminus \bigcup_n A_n$ is null and $f$ agrees with $g_n$ over $A_n$ for every $n$. As $f$ has Luzin's (N), also $f[\mathbb{R}\setminus \bigcup_n A_n]$ is null, and thus can be ignored for our argument. For $y \notin f[\mathbb{R}\setminus \bigcup_n A_n]$, we find that $f^{-1}(y) \subseteq \bigcup_{n \in \mathbb{N}} g_n^{-1}(y)$. Since each set in the right-hand union is countable for almost all $y$ by Corollary \ref{corr:banach}, the union itself is countable for almost all $y$, proving the claim.
\end{proof}
\end{corollary}

{Note that the Borelness of the set $A$ in the corollary above cannot be replaced by ``arbitrary set", as it is consistent with $\mathrm{ZFC}$ that the corresponding statement is false. For example, assuming the continuums hypothesis ($\mathrm{CH}$) or the even weaker Martin's axiom ($\mathrm{MA}$) suffices to construct a counterexample. We do not know whether the following proposition can be proved within $\mathrm{ZFC}$.}

\begin{proposition}[$\mathrm{ZFC}+\mathrm{MA}$]\label{proposition: weird luzin n function}
There is a function $f : \uint \to \uint$ having Luzin's (N) and a set $A \subseteq \uint$ such that $f|_{A}$ is a computable, $f[A]$ is non-null, and for any $y \in f[A]$ the set $f^{-1}(\{y\}) \cap A$ is uncountable.
\begin{proof}
We actually need only a weaker condition than $\mathrm{MA}$ for our construction, namely the equality $\operatorname{cof}(\mathcal{L}) = \operatorname{cov}(\mathcal{L}) = \operatorname{non}(\mathcal{L})$ in Cicho\'n's diagram. Recall that $\operatorname{cof}(\mathcal{L})$ is the least cardinality of a set $R$ of null sets such that any null set is a subset of an element of $R$; $\operatorname{cov}(\mathcal{L})$ is the least cardinal $\alpha$ such that $\uint$ is a union of $\alpha$-many null sets, and $\operatorname{non}(\mathcal{L})$ is the least cardinal of a non-null set. {It is a consequence of $\mathrm{MA}$ that all these cardinals are $2^{\aleph_0}$. As they all are clearly uncountable and at most the continuum, $\mathrm{CH}$ trivially implies the same.} Let $\kappa$ denote the value of these three invariants.

First, we observe that $\kappa = \operatorname{cof}(\mathcal{L})$ means that there exists a family $(z_\alpha)_{\alpha < \kappa}$ such that a set $A \subseteq \Cantor$ is null iff $A \subseteq \mathrm{MLR}(z_\alpha)^C$ for some $\alpha < \kappa$. Next, we point out that $\kappa = \operatorname{cov}(\mathcal{L})$ means that for any $\alpha < \kappa$ and family $(w_\beta)_{\beta < \alpha}$ there exists some $u$ which is Martin-L\"of random relative to all $w_\beta$.

We start with a family $(z_\alpha)_{\alpha < \kappa}$ as above, and then choose $(x_\alpha)_{\alpha < \kappa}$ such that each $x_\alpha$ is Martin-L\"of random relative to any $z_\beta$ for $\beta \leq \alpha$. We then choose another sequence $(y_\alpha)_{\alpha < \kappa}$ such that $y_\alpha$ is Martin-L\"of random relative to any $x_\beta \oplus z_\gamma$ for $\beta, \gamma \leq \alpha$. We identify $\Cantor$ with a positive measure subset of $\uint$ (a fat Cantor set), and then define $A = \{y_\beta \oplus x_\alpha \mid \alpha \leq \beta < \kappa\}$, and $f : \uint \to \uint$ as $f(y_\beta \oplus x_\alpha) = x_\alpha$ and $f(w) = 0$ for $w \notin A$.

As $f|_A$ is just the projection, it is clear that it is computable. To see that $f[A]$ is non-null, note that if it were null, it would need to be contained in $\mathrm{MLR}(z_\alpha)^C$ for some $\alpha < \kappa$. But $x_\alpha \in f[A]$ is explicitly chosen to prevent this. For any $x_\alpha \in f[A]$, we find that $f^{-1}(\{x_\alpha\}) \cap A = \{y_\beta \oplus x_\alpha \mid \alpha \leq \beta < \kappa\}$ has cardinality $\kappa$, and $\kappa$ is uncountable.

It remains to argue that $f$ has Luzin's (N). As $f$ is constant outside of $A$, we only need to consider null sets $B \subseteq A$. Again invoking van Lambalgen's theorem, we see that any $y_\beta \oplus x_\alpha$ is Martin-L\"of random relative to $z_\gamma$ whenever $\gamma \leq \alpha \leq \beta$. As such, each null $B \subseteq A$ is contained in some $\{y_\beta \oplus x_\alpha \mid \alpha \leq \beta < \gamma\}$ for $\gamma < \kappa$. It follows that $f[B] \subseteq \{x_\alpha \mid \alpha < \gamma\}$ has cardinality strictly below $\kappa$, and hence is null due to $\kappa = \operatorname{non}(\mathcal{L})$.
\end{proof}
\end{proposition}

That all fibers are countable is just preservation of $h$-degrees:

\begin{observation}
\label{obs:hypdegree}
For computable $f :  \mathbb{R} \to \mathbb{R}$ the following are equivalent:
\begin{enumerate}
\item $f$ preserves $h$-degrees, i.e.~$\forall x \in \uint \ x \equiv_h f(x)$.
\item For all $y \in \mathbb{R}$, $f^{-1}(y)$ is countable.
\end{enumerate}
\end{observation}

\subsection{Luzin's (N) and higher Kurtz randomness reflection}\label{sec:kurtz}

In his thesis \cite{Luzin1951}, {Luzin showed that if a continuous function
$f:\mathbb R \rightarrow \mathbb R$ fails to have property $(N)$, then
in fact there is a compact witness to this failure.  For the reader's
convenience, we give a proof of this fact below.}
\begin{proposition}
\label{prop:closedsetsuffice}
Let {$f : \mathbb R \to \mathbb R$} be continuous and map some null set to a non-null set. Then there is a {compact subset $A \subseteq \mathbb R$} with $\lambda(A) = 0$ and $\lambda(f(A)) > 0$.
\begin{proof}
{Observe that a function $f:\mathbb R\rightarrow \mathbb R$ satisfies Luzin's (N)
if and only if its restriction $f\upharpoonright [a,b]$ satisfies Luzin's (N) for every
closed interval $[a,b]$.  So without loss of generality, we assume
that $f$ fails Lusin's (N) because $\mu(A) = 0$  but $\mu(f(A)) = d > 0$
for some $A \subseteq [a,b]$.}  As every null set is contained in a $\mathbf{\Pi}^0_2$-null set, without loss of generality we can assume  $A = \cap_n U_n$  for some decreasing sequence of open sets $U_n$ . Each $U_n$ is itself equal to an increasing union of closed sets. The idea is by picking big enough closed $F_n \subseteq U_n$ , we can find a closed set $\bigcap_{n \in \mathbb{N}} F_n \subseteq A$  whose image still has positive measure.  How large to pick the $F_n$?  Let $F_0 \subseteq U_0$ be large enough that $\mu(f(F_0\cap A)) > d/2$.  In general, if we have found $(F_i)_{i<n}$ such that $\mu(f(\cap_{i<n} F_i \cap A)) > d/2$, then since $A \subseteq U_n$, we can find closed $F_n \subseteq U_n$ large enough that  $\mu(f(\cap_{i<n} F_i \cap F_n \cap A))>d/2$  as well.  Therefore for all  $n$, we have $\mu(f(\cap_{i<n} F_i)) > d/2$, and therefore $\mu(\cap_n f(\cap_{i<n} F_i)) \geq d/2$.  Claim:  $\cap_n f(\cap_{i<n} F_i) = f(\cap_n F_n)$.  One direction is clear.  In the other, suppose that  $y \in \cap_n f(\cap_{i<n} F_i)$.  Then $\cap_{i<n} F_n \cap f^{-1}(y)  \neq \emptyset$ for all $n$.  By compactness, $\cap_n F_n \cap f^{-1}(y) \neq \emptyset$.
\end{proof}
\end{proposition}

{As the image of as images of compact sets under continuous functions are uniformly compact,
this shows that Luzin's (N) implies the reflection of all kinds of Kurtz randomness.
This is in contrast to the situation for Martin-L\"of randomness, because the
image of a $\mathbf \Pi^0_2$ set under a continuous $f$ is not even $\mathbf \Pi^0_2$
in general, let alone with the same oracle.}
\begin{proposition}\label{prop:allkurtzreflection}
{If $f:\mathbb R \rightarrow \mathbb R$ satisfies Luzin's (N), then $f$ reflects
Kurtz randomness relative to every oracle.}
\end{proposition}
\begin{proof}
{ Given oracle $Z$,
suppose that $x$ is not $Z$-Kurtz random because $x \in F$, where
$F$ is a $Z$-computable compact set of measure 0.  Then $f(F)$
is also a $Z$-computable compact set, which has measure 0 because
$f$ has Luzin's (N).  Therefore, $f(x)$ is also not $Z$-Kurtz random.
}
\end{proof}
{An immediate consequence is that any $f$ with Luzin's (N)
also reflects $\Delta^1_1$-Kurtz randomness
relative to every oracle.  }
{In general, Kurtz randomness reflection for stronger oracles implies it for weaker ones.}
\begin{proposition}\label{prop:krr_implication}
{If a continuous function $f:\mathbb R \rightarrow \mathbb R$
reflects $Z$-Kurtz randomness, then $f$ reflects $X$-Kurtz randomness
for every $X \leq_T Z$.}
\end{proposition}
\begin{proof}
{Assume that $f$ reflects $Z$-Kurtz randomness.
Suppose $x$ is not $X$-Kurtz random.  Let $P$ be an $X$-computable
compact null set with $x \in P$.  Then $f(P)$ is an $X$-computable
compact set, which is null because $P$ is also $Z$-computable.
Therefore, $f(x)$ is not $X$-Kurtz random.}
\end{proof}
{Additionally, any witness to the failure of
Luzin's (N) also provides an oracle relative to which Kurtz
randomness reflection fails.}
\begin{proposition}\label{prop:krr_implies_n}
{Suppose that $f:\mathbb R \rightarrow \mathbb R$ and $A \subseteq \mathbb R$
is a $Z$-computable compact null set with $\lambda(f(A))>0$.  Then
$f$ does not reflect $Z$-Kurtz randomness.}
\end{proposition}
\begin{proof}
{Since $f(A)$ is also $Z$-computable and has positive measure, it must contain
 some $Z$-Kurtz} random $y$. There is some $x \in A$ with $f(x) = y$,
 but since $A$ is a {$Z$-computable compact} null set, it cannot contain
 any {$Z$-Kurtz} randoms. Hence, $f$ does not reflect {$Z$-Kurtz} randomness.
\end{proof}
{We can thus characterize Luzin's (N) in terms
of Kurtz randomness reflection.}
\begin{theorem}
\label{theo:luzinandkurtz}
The following are equivalent for computable {$f : \mathbb R \to \mathbb R$:}
\begin{enumerate}
\item $f$ has Luzin's (N).
\item For every $\KO$-computable compact set $A$ with $\lambda(A) = 0$ also $\lambda(f(A)) = 0$.
\item $f$ reflects $\KO$-Kurtz randomness.
\end{enumerate}
\begin{proof}
\begin{description}
\item[$1. \Rightarrow 2.$] Trivial.
\item[$2. \Rightarrow 1.$] We observe that given computable {$f : \mathbb R \to \mathbb R$} and number $n$, the set $$\{A \subseteq \mathbb [-n,n] \text{ compact } \mid \lambda(A) = 0 \wedge \lambda(f(A)) \geq 2^{-n}\}$$ is a $\Pi^0_2$-subset of {the Polish space of compact
subsets of $[-n,n]$}. By Proposition \ref{prop:closedsetsuffice}, if $f$ fails Luzin's (N), this set is non-empty for some $n$. If it is non-empty, {it must have an $\KO$-computable element by Kleene's basis theorem.}
\item[$1. \Rightarrow 3.$] {By Proposition} \ref{prop:allkurtzreflection}.
\item[$3. \Rightarrow 2.$] {By Proposition} \ref{prop:krr_implies_n}.
\end{description}
\end{proof}
\end{theorem}

{We also see that $\Delta^1_1$-Kurtz randomness reflection does not suffice
for a characterization.}

\begin{lemma}
\label{lemma:kurtzreflectionsigma11}
Reflecting $\Delta^1_1$-Kurtz randomness is a $\Sigma^1_1$-property of continuous {$f : \mathbb R \to \mathbb R$}.
\begin{proof} {By Proposition}  \ref{prop:krr_implies_n},
reflecting $\Delta^1_1$-Kurtz randomness is equivalent to the statement that for any $\Delta^1_1$ {compact} set $A$ with $\lambda(A) = 0$ {we have} $\lambda(f(A)) = 0$. By Kleene's HYP quantification theorem \cite{kleene4,kleene5}, a universal quantification over $\Delta^1_1$ can be replaced by an existential quantification over Baire space. That $\lambda(A) = 0$ implies $\lambda(f(A)) = 0$ is a $\Delta^0_3$-statement for given $f$ and $A$.
\end{proof}
\end{lemma}

\begin{corollary}\label{cor:Delta11kurtzweaker}
Reflecting $\Delta^1_1$-Kurtz randomness is strictly weaker than Luzin's (N).
\begin{proof}
{By Proposition} \ref{prop:allkurtzreflection}, Luzin's (N) implies $\Delta^1_1$-Kurtz randomness reflection. By Lemma \ref{lemma:kurtzreflectionsigma11} reflecting $\Delta^1_1$-Kurtz randomness is a $\Sigma^1_1$-property. But it was shown in \cite{holicky} that Luzin's (N) is $\Pi^1_1$-complete for continuous functions. Thus the two notions cannot coincide.
\end{proof}
\end{corollary}

\subsection{Open questions}
\label{subsec:questions}
Theorems \ref{theo:main} and \ref{theo:luzinandkurtz} tell us that $\Delta^1_1(\KO)$-randomness reflection {and $\KO$-Kurtz randomness reflection each characterize} Luzin's (N) for computable functions. This does not rule out that other kinds of randomness reflection could also characterize Luzin's (N). {In the next section we shall see that none of MLR-reflection, W2R-reflection, or MLR$(\emptyset')$-reflection imply Luzin's (N) for arbitrary computable functions} (Corollary \ref{corr:separating}).  {Because reflection asks for the
same level of randomness on both sides, there are no completely trivial implications
between the $\mathbf \Pi^0_2$-type
randomness reflection notions.  Indeed, results in} \cite{merkle3} {suggest
that the implication structure between $\mathbf \Pi^0_2$-type
randomness reflection notions may
have little relation to the implication structure between notions of randomness.
However,} the most interesting open question seems to be:

\begin{question}
\label{question:delta11}
Can a computable function reflect $\Delta^1_1$-randomness but fail Luzin's (N)?
\end{question}

By Theorem \ref{theo:main} any such example would need to have a positive measure of fibers being uncountable, which is incompatible with most niceness conditions. {
We also do not know the answer to the above question if
$\Delta^1_1$-randomness is replaced with Martin-L\"of
randomness relative to $\emptyset^{(\alpha)}$ for any $\alpha \geq 2$.}

{Related questions concern basis theorems for failures of Luzin's (N).  We
have already seen in Theorem} \ref{theo:luzinandkurtz} {that any computable
$f$ which fails Luzin's (N) must see that failure witnessed by a $\mathcal O$-computable
compact set.  The proof shows that such a set can also be chosen hyperarithemetically
low, by applying Gandy basis theorem in place of the Kleene.
On the other hand, Corollary }
\ref{cor:Delta11kurtzweaker} {shows that a function which fails Luzin's (N)
need not have a hyperarithmetic compact witness.  Indeed, one can obtain specific examples
of this separation by feeding pseudo-well-orders into the $\Pi^1_1$-completeness
construction of }\cite{holicky}.  {Thus the results for compact
witnesses are rather tight overall.}

{The situation for the minimum complexity of
$\mathbf \Pi^0_2$ witnesses is less well understood.  The proof of
Corollary }\ref{corr:separating} {shows that a computable
function may fail
Luzin's (N) while still mapping all rapidly null $\Pi^0_2(\emptyset')$
sets
to null sets.  That is, the set $\mathrm{MLR}(\emptyset')^C$ is mapped to a null set.}

\begin{question}
\label{question:pi02basis}
{Can a computable function map $W3R^C$
to a null set but fail Luzin's (N)?  Equivalently, can
a computable function map all null $\Pi^0_2(\emptyset')$
sets to null sets while failing Luzin's (N)?}
\end{question}

{We note that the functions produced by the $\Pi^1_1$-completeness
construction of } \cite{holicky} {are of no help because
the failure of Lusin's (N), when it occurs,
is witnessed by an effectively null $\Pi^0_2$ set.}

\section{Separating Luzin's (N) from MLR-reflection}
\label{sec:separating}
We present a construction of a computable function that violates Luzin's (N),
{and yet is piecewise-linear in a neighborhood of every point that is not $\mathrm{MLR}(\emptyset')$.} {Here, we say that $f$ is piecewise-linear in a neighborhood of $x$, if there are rationals $a < x < b$ such that $f|_{[a,x]}$ and $f|_{[x,b]}$ are linear functions. Computable piecewise-linear functions reflect essentially all kinds of randomness.}

\begin{theorem}
\label{theo:separating}
For each {$\Pi^0_1(\emptyset')$}-set $A \subseteq \uint$
there is a computable function $f : \uint \to \uint$ such that:
\begin{enumerate}
\item For every $x \in \uint \setminus A$, $f$ is {piecewise-linear} on a neighbourhood of $x$.
\item {For every $\varepsilon>0$}, there is a null {$\Pi^0_1(\emptyset'')$} set $B \subseteq A$ such that $\lambda(f[B]) \geq \lambda(A) - \varepsilon$.
\end{enumerate}
\end{theorem}

\begin{corollary}
\label{corr:separating}
There is a computable function $f : \uint \to \uint$ that reflects ML-randomness, weak-2-randomness and ML($\emptyset'$)-randomness, yet does not have Luzin's (N), nor reflects {weak-3-randomness}.
\begin{proof}
Let $A$ be the complement of the first component of a universal ML($\emptyset'$)-test. Then $\lambda(A) > \frac{1}{2}$. We invoke Theorem \ref{theo:separating} on $A$ and $\epsilon = \frac{1}{4}$. The resulting function is the desired one: If $x \in \uint$ is not ML($\emptyset'$)-random, then $x \notin A$, $f$ is {piecewise-linear} {on a neighborhood of} $x$, and thus $f(x)$ is not ML($\emptyset'$)-random. As such, we conclude that whenever $f(x)$ is ML($\emptyset'$)-random, then so is $x$ (same for the other notions).

Since we can choose the witness $B$ as being {$\Pi^0_1(\emptyset'')$}, it is also {$\Pi^0_2(\emptyset')$}, and thus contains only elements which are not {weak-3-random}. Since $f[B]$ has positive measure, it contains a {weak-3-random} -- hence $f$ does not reflect {weak-3-randomness}.
\end{proof}
\end{corollary}

{We remark that this is the strongest result possible for the strategy
we are using.  We are making sure that $f$ reflects
$\mathrm{MLR}(\emptyset')$-randomness by making $f$ piecewise-linear in
a neighborhood of every non-$\mathrm{MLR}(\emptyset')$ point. However, the following proposition shows
that the set of points where $f$ can
be this simple has a descriptive complexity of $\Sigma^0_1(\emptyset')$.
But the weak-3-non-randoms are not contained in any $\Sigma^0_1(\emptyset')$ set except $\uint$.}

\begin{proposition}
Let $f : \mathbb{R} \to \mathbb{R}$ be computable. The set of points where $f$ is locally piecewise-linear is $\Sigma^0_1(\emptyset')$.
\begin{proof}
We consider the property $\mathrm{2L}$ of a function $f$ and an interval $[a,b]$ that there is some $x \in [a,b]$ such that both $f|_{[a,x]}$ and $f|_{[x,b]}$ are linear. We first claim that this is a $\Pi^0_1$-property. To this, we observe that $\mathrm{2L}$ is equivalent to: {\small \[\forall n \in \mathbb{N} \exists i \leq n \quad \left (\forall j \in \{0,\ldots,n-2\} \setminus \{i-1,i,i+1\} \ f(a + \frac{j}{n}) - f(a + \frac{j+1}{n}) = f(a + \frac{j+1}{n}) - f(a + \frac{j+2}{n}) \right )\] }

Next, we observe that $f$ is locally piecewise-linear in $x$ iff there is some rational interval $(a,b) \ni x$ such that $f$ has property $\mathrm{2L}$ on $[a,b]$. Using $\emptyset'$, we can enumerate all these intervals.
\end{proof}
\end{proposition}

Generalizing this idea slightly, recall that a function is bi-Lipschitz, if both the function and its inverse are Lipschitz functions, i.e.~if there exists some constant $L$ such that $d(f(x),f(y)) \leq Ld(x,y) \leq L^2d(f(x),f(y))$ for all $x, y$ in the domain. Since computable locally bi-Lipschitz functions preserve and reflect all kinds of randomness,
Another way for $f$ to ensure a given notion of randomness
reflection is by being locally bi-Lipschitz on the
non-random points for that notion. However, we still get a $\Sigma^0_1(\emptyset')$-set of suitable points.

\begin{proposition}
Let $f : \mathbb{R} \to \mathbb{R}$ be computable. The set of points where $f$ is locally bi-Lipschitz is $\Sigma^0_1(\emptyset')$.
\begin{proof}
The following is a co-c.e.~property in $a,b \in \mathbb{Q}$ and $L \in \mathbb{N}$ and $f \in \mathcal{C}(\mathbb{R},\mathbb{R})$: $$\forall x,y \in [a,b] \quad d(x,y) \leq Ld(f(x),f(y)) \leq L^2d(x,y)$$
We obtain the set of points where $f$ is locally bi-Lipschitz by taking the union of all $(a,b)$ having the property above for some $L \in \mathbb{N}$ -- access to $\emptyset'$ suffices to get such an enumeration.
\end{proof}
\end{proposition}

\subsection{High-level proof sketch}
{Before diving into the details of the proof of Theorem}
\ref{theo:separating} {we give a high-level sketch of
what is going on.  Consider first the case where $A$ is
$\Pi^0_1$.  Then $A = \cap_n A_n$, where each $A_n$
is a finite union of closed intervals and $A_{n+1}\subseteq A_n$.
We iteratively define
a sequence of piecewise linear
functions $f_0,f_1,\dots$, where
$f_0:\uint \rightarrow \uint$ is
the identity, and $f_{n}$ is obtained from $f_{n-1}$ by performing
a ``tripling'' operation on those line segments of $f_{n-1}$ which
are contained in $A_{n}$.  In order to ``triple'' a line segment,
we replace it by a zig-zag of three line segments each of which
has triple the slope of the original.  (See Figure }
\ref{fig:constructionzigzag} {)  We want to make the
sequence $(f_n)$ converge in the supremum norm, so before
tripling we add invisible break points to $f_{n-1}$ so that none
of its linear pieces are more than $2^{-n}$ tall.  Letting $f$
be the limit function, observe that if $x \not\in A_n$, then
$f$ coincides with $f_n$ on a neighborhood of $x$,
and thus $f$ is linear on a neighborhood of $x$.  On the other
hand, $A$ is then exactly the set of points where we
tripled infinitely often.}

{Next we describe how to find a closed set $B\subseteq A$
such that $\mu(f(B)) \geq \mu(A)$.  (The $\varepsilon$ in the
statement of the theorem exists in order to bring down the
descriptive complexity of $B$, but we can ignore it for now.)
We want $B\subseteq A$, so of course we throw out of $B$
any interval that leaves $A$.  Also, every time we perform a tripling,
we choose two-thirds of the tripled interval to throw out of $B$.
We do this so that the one-third which we keep has maximal
measure of intersection with $A$.  Observe that $\mu(B) = 0$.
}

{Here is why $\mu(f(B))\geq \mu(A)$.  Let $B_0 = [0,1]$
and let $B_n$ denote
the set of points that remain in $B$ at the end of stage $n$.
By induction, $f_n \upharpoonright B_n$ is
injective (except possibly at break points) and
$\mu(f_n(B_n\cap A)) \geq \mu(A)$.  The key to the
induction is that by the choice of thirds, we always have
$3^n\mu(B_n\cap A) \geq \mu(A)$, and since $f_n$ has
slope $\pm3^n$ on all of $B_n$ and is essentially injective,
$\mu(f_n(B_n \cap A)) = 3^n\mu(B_n\cap A)$.  It now
follows that $\mu(f_n(B_n)) \geq \mu(A)$ for all $n$.
Furthermore, since the continuous image of a
compact set is uniformly compact, we cannot
have $\mu(f(B)) < \mu(A)$,
for this would have been witnessed already for some
$\mu(f_n(B_n))$.  This completes the sketch for the
case where $A$ is $\Pi^0_1$.
}

{If $A$ is $\Pi^0_1(\emptyset')$, we can do
essentially the same construction, tripling on
the stage-$n$ approximation to $A_n$ instead of
$A_n$ itself.  Any interval which is going to leave $A$
eventually leaves the approximations for good, so
the key features of the above argument are maintained
even as the structure of the triplings gets more complicated.}

\subsection{Proof of Theorem \ref{theo:separating}}
The remainder of this section is devoted to the preparation for the proof of Theorem \ref{theo:separating}, and the proof itself.

\begin{lemma}
\label{lemma:shrinking}
Given a $\Pi^0_1(\emptyset')$-set $A \subseteq \uint$ and some open $U \supseteq A$ we can compute some open $V$ with $A \subseteq V \subseteq U$ such that $d(V,U^C) > 0$.
\begin{proof}
Since $U^C$ and $A$ are disjoint closed sets, there is some $N \in \mathbb{N}$ such that $d(U^C,A) > 2^{-N}$. If we actually had access to $A$, we could compute a suitable $N$. Since $A$ is computable from $\emptyset'$, we can compute $N$ with finitely many mindchanges. The monotonicity of correctness here means we can actually obtain suitable $N \in \mathbb{N}_<$. We now obtain $V$ by enumerating an interval $(a,b)$ into $V$ once we have learned that $U$ covers $[a-2^{-N},b+2^{-N}]$ (which is semidecidable in $U \in \mathcal{O}(\mathbb{R})$ and $N \in \mathbb{N}_<$).
\end{proof}
\end{lemma}

For an interval $[a,b]$, let $T_0([a,b]) = [a, a + \frac{b-a}{3}]$, $T_1([a,b]) = [a + \frac{b-a}{3}, a + 2\frac{b-a}{3}]$ and $T_2([a,b]) = [a + 2\frac{b-a}{3}, b]$.

\begin{lemma}
\label{lem:convenientsequence}
Let $A$ be a $\Pi^0_1(\emptyset')$ set. Then there is a computable double-sequence $(I^k_n)_{k,n \in \mathbb{N}}$ of closed intervals with the following properties:
\begin{enumerate}
\item $A = \bigcap_{n \in \mathbb{N}} \bigcup_{k \in \mathbb{N}} I^k_n$.
\item $I^k_n$ and $I^\ell_n$ intersect in at most one point.
\item For {$m < n$}, we find that $\bigcup_{k \in \mathbb{N}} I^k_n$ has positive distance to the complement of $\bigcup_{k \in \mathbb{N}} I^k_m$.
\item $\forall k, n \in \mathbb{N} \quad |I^k_n| \geq |I^{k+1}_n|$
\item Fix $n > 0$. For each $k$ there are $\ell$, $i$ such that $|I^k_{n}| < 3^{-2}|I^\ell_{n-1}|$ and $I^k_n \subseteq T_i(I^\ell_{n-1})$.
\end{enumerate}
\begin{proof}
Any $\Pi^0_1(\emptyset')$ is in particular $\Pi^0_2$, and thus has $\Pi^0_2$-approximation $A = \bigcap_{n \in \mathbb{N}} U_n$. We invoke Lemma \ref{lemma:shrinking} inductively first on $A$ and $U_0$ to obtain $V_0$, then on $A$ and $U_1 \cap V_0$ to obtain $V_1$, and so on. This will ensure Condition (3). We can effectively write any open set $V_n \subseteq \uint$ as a union of closed intervals such that the pairwise intersections contain at most one point. To make Conditions (4,5) work it suffices to subdivide intervals sufficiently much.
\end{proof}
\end{lemma}

\begin{definition}
An interval $J$ is \emph{well-located} relative to $(I^k_n)_{k,n \in \mathbb{N}}$, if for {all} $k, n$ one of the following hold:
\begin{enumerate}
\item $|J \cap I^k_n| \leq 1$
\item $J \supseteq I^k_n$
\item $J \subseteq T_i(I^k_n)$ for some $i \in \{0,1,2\}$
\end{enumerate}
For well-located $J$, let its depth be the greatest $n$ such that $J \subseteq I^k_n$ for some $k$. We call two well-located intervals $J_0, J_1$ {peers}, if whenever $J_b \subsetneq I^k_n$ for both $b \in \{0,1\}$ and some $k, n$, then there is one $i \in \{0,1,2\}$ such that $J_b \subseteq T_i(I^k_n)$ for both $b \in \{0,1\}$.
\end{definition}

Note that our requirements for the $(I^k_n)_{k,n \in \mathbb{N}}$ in Lemma \ref{lem:convenientsequence} in particular ensure that each $I^{k_0}_{n_0}$ is well-located relative to $(I^k_n)_{n,k \in \mathbb{N}}$.

\begin{definition}
\label{def:core}
We are given a double-sequence $(I^k_n)_{k,n \in \mathbb{N}}$ for a set $A$ as in Lemma \ref{lem:convenientsequence} and $\varepsilon > 0$. Let $N_n \in \mathbb{N}$ be chosen sufficiently large such that $\lambda(\bigcup_{k > N_n} I_n^k) < 3^{-n}2^{-n-2}\varepsilon$. Let $b_{k,n} \in \{0,1,2\}$ be chosen such that $\lambda(T_{b_{k,n}}(I^k_n) \cap A) + 3^{-n}2^{-n-3-k}\varepsilon \geq \lambda(T_{c}(I^k_n) \cap A)$ for all $k, n \in \mathbb{N}$ and $c \in \{0,1,2\}$. Let $\ell_{n,k}$ and $i_{n,k}$ be the witnesses for Condition 5 in Lemma \ref{lem:convenientsequence}. We then inductively define $\mathfrak{I}_0^\varepsilon = \{I_0^k \mid k \leq N_0\}$ and: $$\mathfrak{I}_n^\varepsilon = \{I_n^k \mid k \leq N_n \wedge I^{\ell_{n,k}}_{n-1} \in \mathfrak{I}_{n-1} \wedge b_{\ell_{n,k},(n-1)} = i_{n,k}\}$$
\end{definition}

In words, the intervals in $\mathfrak{I}_n^\varepsilon$ are those on the $n$-th level which occur inside a particular third of their parent intervals on the $n-1$-st level which has the maximal measure of its intersection with the set $A$. By construction, the intervals in $\mathfrak{I}_n$ are pairwise {peers}.

\begin{lemma}
\label{lemma:pacomplete}
Starting with a $\Pi^0_1(\emptyset')$-set $A$, we can compute the sets $\mathfrak{I}_n^\varepsilon$ relative to $\emptyset''$.
\begin{proof}
As the double-sequence $(I^k_n)_{k,n \in \mathbb{N}}$ is computable, we can obtain the sufficiently large $N_n$ by using $\emptyset'$. We have $T_{c}(I^k_n) \cap A$ available to us as $\Pi^0_1(\emptyset')$-sets, so $\emptyset''$ lets us compute $\lambda(T_{c}(I^k_n) \cap A) \in \mathbb{R}$. Then getting the choices for the $b_{k,n}$ right can be done computably. The witnesses $\ell_{n,k}$,$i_{n,k}$ can also be found computably.
\end{proof}
\end{lemma}

\begin{lemma}
\label{label:measurebounds}
$3^{-n} \geq \lambda(\bigcup \mathfrak{I}_n^\varepsilon) \geq 3^{-n}(\lambda(A)-\varepsilon)$
\begin{proof}
We prove both inequalities by induction. For the first, the base case is trivial. For the induction step, we note that $(\bigcup \mathfrak{I}_{n+1}^\varepsilon) \subseteq \bigcup_{I^k_n \in \mathfrak{I}_n^\varepsilon} T_{b_{k,n}}(I^k_n)$, and that $\lambda \left (\bigcup_{I^k_n \in \mathfrak{I}_n^\varepsilon} T_{b_{k,n}}(I^k_n) \right ) = \frac{1}{3} \lambda(\bigcup \mathfrak{I}_n^\varepsilon)$.

For the second inequality, we prove a stronger claim, namely that $\lambda(A \cap \bigcup \mathfrak{I}_n^\varepsilon) \geq 3^{-n}(\lambda(A)-(1-2^{-n-1})\varepsilon)$. The base case follows from $\bigcup_{k \in \mathbb{N}} I_0^k \supseteq A$ together with the choice of $N_0$. We then observe that $A \cap (\bigcup \mathfrak{I}_{n+1}^\varepsilon) = A \cap \bigcup \{I^k_{n+1} \mid \exists I^\ell_n \in \mathfrak{I}_n^\varepsilon \ \exists i \in \{0,1,2\} \ \ I^k_{n+1} \subseteq T_i(I^\ell_n)\}$. By definition of $b_{\ell,n}$ in Definition \ref{def:core}, this also means that $\lambda(A \cap (\bigcup \mathfrak{I}_{n+1}^\varepsilon)) + 3^{-n}2^{-n-2}\varepsilon \geq 3^{-1}\lambda(A \cap \bigcup \{I^k_{n+1} \mid \exists I^\ell_n \in \mathfrak{I}_n^\varepsilon \ \ I^k_{n+1} \subseteq T_{b_{\ell,n}}(I^\ell_n)\})$. The set on the right hand side differs from $\bigcup \mathfrak{I}_n^\varepsilon$ only by the fact that in the latter, we restrict to $\ell \leq N_n$. By the induction hypothesis together with the guarantee that $\lambda(\bigcup_{k > N_n} I_n^k) < 3^{-n}2^{-n-2}\varepsilon$ we get the desired claim.
\end{proof}
\end{lemma}

\begin{lemma}
\label{lemma:locality}
For a sequence $(I^k_n)_{k,n \in \mathbb{N}}$ as in Lemma \ref{lem:convenientsequence} and $x \notin A$, it holds that there exists some $a < x < b$ and some $N \in \mathbb{N}$ such that $I^k_n \cap (a,b) = \emptyset$ for every $n > N$.
\begin{proof}
If $x \notin A$, then there is some $N$ with $x \notin \bigcup_{k \in \mathbb{N}} I^k_N$. By Condition (3) of Lemma \ref{lem:convenientsequence}, we have that $x$ has positive distance to $\bigcup_{k \in \mathbb{N}} I^k_{N+1}$, hence there exists an interval $(a,b)$ around $x$ disjoint from $\bigcup_{k \in \mathbb{N}} I^k_{N+1}$, and by monotonicity, also from $\bigcup_{k \in \mathbb{N}} I^k_{n}$ for every $n > N$.
\end{proof}
\end{lemma}

\begin{proof}[Proof of Theorem \ref{theo:separating}]
{\bf Preparation:} We note that for each $\Pi^0_2$-set $A$ and each $n \in \mathbb{N}$, there is a $\Pi^0_1(\emptyset')$-set $C$ with $C \subseteq A$ and $\lambda(C) \geq \lambda(A) - 2^{-n}$. We can assume w.l.o.g. that $A$ is already $\Pi^0_1(\emptyset')$. We then obtain a computable double-sequence $(I^k_n)_{k,n \in \mathbb{N}}$ as in Lemma \ref{lem:convenientsequence}.

{\bf Construction:} We obtain our function $f$ as the limit of functions $f_{N,K}$ for $N,K \in \mathbb{N}$. $f_{0,0}$ is the identity on $\uint$. The construction of $f_{N,K}$ takes into account only intervals $I^k_n$ with $n \leq N$ and $|I^k_n| > 2^{-K}$, of which there are only finitely many (and by monotonicity of the enumerations, we can be sure that we have found them all). We process intervals with smaller $n$ first, and replace the linear growth $f$ currently has on $I^k_n$ by a triple as shown in Figure~\ref{fig:constructionzigzag}. Property 5 from Lemma \ref{lem:convenientsequence} ensures that through the process, the function is linear on each interval $I^k_n$ yet to be processed. We then define $f_N := \lim_{K \to \infty} f_{N,K}$ and $f := \lim_{N \to \infty} f_N$.

\begin{figure}[ht]
\centering
\includegraphics[width=0.8\textwidth]{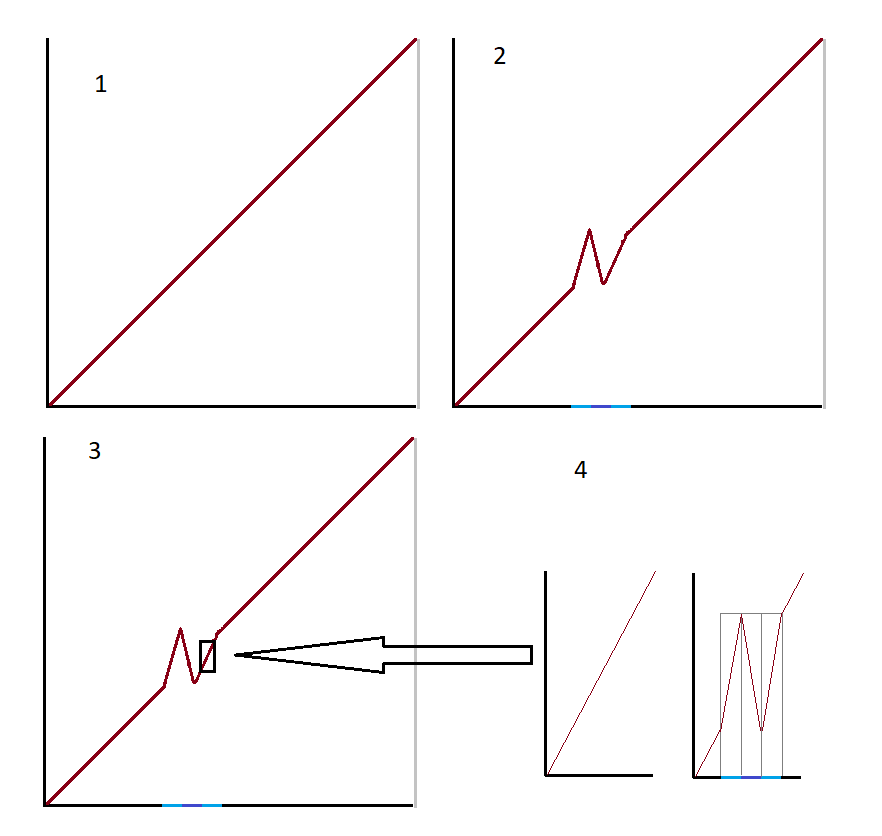}
\caption{Demonstrating the interative construction of the function $f$ in the proof of Theorem  \ref{theo:separating}}
\label{fig:constructionzigzag}
\end{figure}

That the first limit has a computable rate of convergence follows from the monotonicity of $|I^k_n|$ in $k$. Since the size of the intervals shrinks sufficiently fast compared to the potential growth rates of $f_N$, we see that we also do have a computable rate of convergence of $(f_N)_{N \in \mathbb{N}}$.

{\bf Property 1:} If $x \notin A$, we can invoke Lemma \ref{lemma:locality} to obtain a neighbourhood $U$ of $x$ that is disjoint from any $I^k_n$ for $n > N$. But that ensures that $f|_U$ is $3^N$-Lipschitz, and and by potentially restricting the interval further we can make $f$ locally bi-Lipschitz.

\begin{lemma}
\label{lemma:fwell-located}
\begin{enumerate}
\item Let $J$ be well-located at depth $N$. Then $f[J] = f_N[J]$.
\item Let $J$ be well-located at depth $n$. Then $\lambda(f[J]) = 3^n\lambda(J)$.
\item Let $J_0, J_1$ be {peer} well-located intervals. Then $|f[J_0] \cap f[J_1]| \leq 1$.
\end{enumerate}
\begin{proof}
\begin{enumerate}
\item First, we observe that for any $M > N$ it holds that $f_{M}[J] = f_N[J]$, since all modifications based on intervals $I^k_n$ with $n > N$ will affect $f|_J$ at most by locally replacing the shape of the graph with a different shape having the same range. It remains to argue that the identity of the image $f_{M}[J]$ is preserved by limits. Since $\uint$ is compact and Hausdorff, we find that $\mathcal{A}(\uint) \cong \mathcal{K}(\uint)$, hence we can compute $g[J] \in \mathcal{A}(\uint)$ from $g$ and $A \in \mathcal{A}(\uint)$. We have access to $g[A] \in \mathcal{V}(\uint)$ given $A \in \mathcal{V}(\uint)$ anyway. Since $\mathcal{A}(\uint) \wedge \mathcal{V}(\uint)$ is Hausdorff \cite{pauly-locallycompact}, this yields the claim.
\item By (1.), it suffices to show $\lambda(f_n[J]) = 3^n\lambda(J)$ instead. Now $(f_n)|_J$ is just a linear function with slope $3^n$, which yields the claim.
\item If $J_0,J_1$ are {peers} and well-located, and $J_0 \subseteq I^k_n$ but $J_1 \nsubseteq I^k_n$, then $I^k_n$ and $J_1$ are also {peers}. It thus suffices to prove the claim for the case where $J_0 = I^{k_0}_{n+1}$ and $J_1 = I^{k_1}_{n+1}$. These are both contained in the same $T_i(I^\ell_{n})$, and $(f_n)|_{T_i(I^\ell_{n})}$ is a linear function. Since $|J_0 \cap J_1| \leq 1$ it follows that $|f_n[J_0] \cap f_n[J_1]| \leq 1$. By $(1.)$, this already yields the claim.
\end{enumerate}
\end{proof}
\end{lemma}

{\bf Property 2:} We obtain the desired set $B$ as $B = \bigcap_{n \in \mathbb{N}} \left (\bigcup \mathfrak{I}_n^\varepsilon \right )$. Since each $\mathfrak{I}^\varepsilon_n$ is a finite collection of closed intervals, $B$ is indeed closed. Since the intersection is nested and $\lambda(\bigcup \mathfrak{I}_n^\varepsilon) \leq 3^{-n}$ by the first part of Lemma \ref{label:measurebounds}, we conclude that $\lambda(B) = 0$. Since the intervals in $\mathfrak{I}_n^\varepsilon$ are well-located and pairwise {peers}, we know that $\lambda(f([\bigcup \mathfrak{I}_n^\varepsilon])) = 3^n\lambda(\bigcup \mathfrak{I}_n^\varepsilon)$ by Lemma \ref{lemma:fwell-located} 2\&3. Invoking the second inequality from Lemma \ref{label:measurebounds} then lets us conclude $\lambda(f([\bigcup \mathfrak{I}_n^\varepsilon])) \geq (\lambda(A)-\varepsilon)$. Since this estimate holds for every stage of a nested intersection of compact sets, it follows that $\lambda(f[B]) \geq \lambda(A)-\varepsilon$ as desired. That $B$ is obtainable by an oracle of the claimed strength follows from Lemma \ref{lemma:pacomplete}.
\end{proof}

\section{Luzin's (N), absolute continuity and bounded variation}
\label{sec:boundedvariation}

We recall the definitions of absolute continuity and bounded variation:

\begin{definition}
A function $f : \uint \to \mathbb{R}$ is \emph{absolutely continuous}, if {for every
$\varepsilon>0$} and every  $x_0 < y_0 < x_1 < y_1 \ldots < x_k < y_k$ {there is a
$\mathbb \delta >0$ such} that:$$\Sigma_{i \leq k} |y_i - x_i| < \delta \quad \Rightarrow \quad \Sigma_{i \leq k} |f(y_i) - f(x_i)| < \varepsilon.$$
\end{definition}

\begin{definition}
A function $f : \uint \to \mathbb{R}$ has bounded variation, if there is some bound $M \in \mathbb{N}$ such that for any $k \in \mathbb{N}$ and any $x_0 < x_1 < \ldots < x_k$ it holds that $$\Sigma_{i < k} |f(x_{i+1}) - f(x_i)| < M$$
\end{definition}

Being absolutely continuous implies having bounded variation. These notions are related to Luzin's (N) by the following classical fact:
\begin{fact}[see \cite{saks}, Theorem VII.6.7]
\label{fact:absolutevariation}
A {continuous} function is absolutely continuous iff it has both bounded variation and Luzin's (N).
\end{fact}

We observe that being absolutely continuous is a $\Pi^0_3$-property, and recall that Luzin's (N) is $\Pi^1_1$-complete \cite{holicky}. As such, restricting our attention to functions of bounded variation should alter the situation significantly.

\begin{proposition}
\label{prop:weak2random}
If {$f : \uint \to \mathbb{R}$} is computable and absolutely continuous, then $f$ reflects weak-2-randomness.
\begin{proof}
First, we consider how we can exploit connectedness of $\mathbb{R}$ to say something about the images of open sets under computable functions. We are given open sets in the form $U = \bigcup_{i \in \mathbb{N}} I_{i}$, where each $I_{i}$ is an open interval with rational endpoints. We can then compute $\sup f(I_{i})$ and $\inf f(I_{i})$ (as these are equal to $\max f(\overline{I_{n,i}})$ and $\min f(\overline{I_{n,i}})$, and we can compute minima and maxima of continuous functions on compact sets). Let $V = \bigcup_{i \in \mathbb{N}} (\inf f(I_{i}),\sup f(I_{i}))$. We note that we can compute $V$ from $U$, that $V \subseteq f[U]$, and that $f[U] \setminus V$ can only contain computable points. In particular, $\lambda(V) = \lambda(f[U])$.

Now we assume that $f$ additionally is absolutely continuous, and that we are dealing with a $\Pi^0_2$-null set $A = \bigcap_{n \in \mathbb{N}} U_n$ witnessing that some $x \in A$ is not weak-2-random. We assume that $U_{n+1} \subseteq U_n$. As $A$ is null, we know that $\lim_{n \to \infty} \lambda(U_n) = 0$. Since $f$ is absolutely continuous, we also have $\lim_{n \to \infty} f[U_n] = 0$.   Let $V_n$ be obtained from $U_n$ as in the first paragraph, and $B = \bigcap_{n \in \mathbb{N}} V_n$. It follows that $\lambda(B) = 0$, and moreover, $f[A]$ is contained in $B$ with the potential exception of some computable points. So we can conclude that $f(x)$ is not weak-2-random, either because $f(x)$ is computable, or because $f(x) \in B$.
\end{proof}
\end{proposition}

{Note that if we had started with a Martin-L\"of test in the argument above, we
would have no guarantee of ending up with one, because the modulus of absolute continuity
is not computable in general.  Indeed, absolute continuity does not imply MLR reflection.
See Corollary} \ref{cor:n_notimply_mlrr}.

\begin{lemma}
\label{lem:kurtzzerojump}
If $f:[0,1]\rightarrow \mathbb R$ is computable, has bounded variation, and reflects $\emptyset'$-Kurtz
randomness, then $f$ has property $(N)$.
\begin{proof}
Suppose that $f$ does not have $(N)$.
Since $f$ has bounded variation,
it must fail absolute continuity.  Let $\varepsilon>0$ be such that for all $\delta>0$,
there is a finite union of intervals $A_\delta \subseteq[0,1]$ with $\mu(A_\delta)<\delta$
and $\mu(f(A_\delta))>\varepsilon$.
Computably, given $\delta$ we can find such $A_\delta$ by searching.
Let $A = \cap_n U_n$, where $U_n = \cup_{m>n} A_{2^{-m}}$.  Then $A$ is $\Pi^0_2$,
and $\mu(A) = 0$, and $\mu(\cap_n f(U_n)) \geq \varepsilon$.  We
claim that its subset $f(A)$ also has $\mu(f(A)) \geq \varepsilon$.
{Let $\Varf:[0,1]\rightarrow\mathbb R$ denote the cumulative
variation function of $f$, defined by setting $\Varf(x)$ to be equal
to the variation of $f$ on $[0,x]$.}
Since
$f$ has bounded variation and $U_{n+1}\subseteq U_n$,
$\sum_n \Varf(U_{n}\setminus U_{n+1})$ is finite, so by choosing $N$ large enough,
we can make $\sum_{n>N} \mu(f(U_n\setminus U_{n+1}))$ as small as we like.  Now observe
that no matter how large $N$ we choose,
$$\left(\bigcap_n f(U_n) \setminus f(A)\right) \subseteq \bigcup_{n>N} f(U_n \setminus U_{n+1}).$$
This proves the claim.  We have found a $\Pi^0_2$ set $A=\cap_n U_n$ which witnesses
the failure of $(N)$.

Observe that for any c.e. open set $U$, $\mu(f(U))$ is c.e..  Therefore, since $f$
has bounded variation,
$\emptyset'$ can search around to find, for each $n$, a closed set $F_n \subseteq U_n$
such that $\mu(f(U_n\setminus F_n)) < 2^{-n-2}\varepsilon$.  The existence of such a closed set is guaranteed by $f$ having bounded variation.

Let $F = \cap_n F_n$.  Then $F \subseteq A$ and $A \setminus F = \cup_n (A\setminus F_n)$.
So $$\mu(f(A\setminus F)) \leq \sum_n \mu(f(A \setminus F_n)) \leq \sum_n \mu(f(U_n\setminus F_n)) \leq \sum_n 2^{-n-2}\varepsilon < \varepsilon.$$
The positive measure of $f(F)$
then follows as $$\varepsilon \leq \mu(f(A)) \leq \mu(f(A\setminus F)) + \mu(f(F)).$$
Therefore, $F$ is an $\emptyset'$-computable closed set of measure
zero whose image has positive measure.  So $f$ does not reflect $\emptyset'$-Kurtz
randomness.
\end{proof}
\end{lemma}

\begin{theorem}
\label{theo:boundedvariation}
The following are equivalent for computable functions $f : \uint \to \mathbb{R}$ having bounded variation:
\begin{enumerate}
\item $f$ has Luzin's (N).
\item $f$ reflects weak-2-randomness.
\item $f$ reflects $\emptyset'$-Kurtz randomness.
\item $f$ reflects $\Delta^1_1(\KO)$-randomness.
\item $f$ reflects $Z$-Kurtz randomness for any $Z \geq \emptyset'$.
\end{enumerate}
\end{theorem}
\begin{proof}
The implication from $(1)$ to $(2)$ is given by Proposition \ref{prop:weak2random}.
{To see that $(2)$ implies $(3)$, first observe that
 weak-2-randomness reflection implies that
$\mu(f(A)) = 0$ for any null $\Pi^0_2$ set $A$, for if $f(A)$ had positive
measure then it would certainly contain weak-2-random elements.  A
$\Pi^0_1(\emptyset')$ set is in particular $\Pi^0_2$, so the image of any $\emptyset'$-Kurtz
test has measure 0, and is thus also an $\emptyset'$-Kurtz test because the
continuous image of a compact set is uniformly compact.}
The implication $(3) \Rightarrow (1)$ is in Lemma \ref{lem:kurtzzerojump}.

{Finally, the equivalence of (1) and (4) is just Theorem }
\ref{theo:main},{ the implication from (1) to (5) is
Proposition }\ref{prop:allkurtzreflection}, {and the implication from
(5) to (3) is Proposition} \ref{prop:krr_implication}.
\end{proof}

\begin{corollary}
If a computable function $f : \uint \to \mathbb{R}$ of bounded variation reflects ML-randomness, then it has Luzin's (N).
\end{corollary}
{The converse is false; see Corollary} \ref{cor:n_notimply_mlrr}.
\begin{proof}
The same argument works as for the implication {$(2) \Rightarrow (3)$} in Theorem \ref{theo:boundedvariation}.
\end{proof}

{In this section we have stated all results for $f:\uint \rightarrow \mathbb R$
because this is a natural setting in which to consider functions of bounded variation.
Of course, our pointwise
results are equally true for any computable $f:\mathbb R \rightarrow \mathbb R$
which is locally of bounded variation.}

An often useful result about continuous functions of bounded variation is that they can be obtained as difference between two strictly increasing continuous functions. In light of our investigation of Luzin's (N) for strictly increasing functions, one could wonder why we are not exploiting this property here. There are two obstacles: One the one hand, the computable counterpart of the decomposition result is false: There is a computable function of bounded variation, which cannot be written as the difference between any two strictly increasing computable functions \cite{rettinger3}. On the other hand, Luzin's (N) is very badly behaved for sums. For example, for every continuous function $f$ having Luzin's (N) there exists another continuous function $g$ having Luzin's (N) such that $f + g$ fails (N) \cite{pokorny}.

\section{The relationship to absolute continuity of measures}
\label{sec:measures}
For increasing functions we see a connection to absolute continuity of measures. Recall that a measure $\mu$ is absolutely continuous w.r.t.~a measure $\nu$ (in symbols $\mu \ll \nu$), if $\nu(A) = 0$ implies that $\mu(A) = 0$. The notions are related through the following observations:

\begin{observation}
If continuous {surjective} $f : \uint \to \uint$ is increasing, then the {probability}
measure $\mu$ defined as
{$\mu(A) = \lambda(f(A))$} is non-atomic, and {$\mu \ll \lambda$} iff $f$ has Luzin's (N).
\end{observation}

\begin{observation}
\label{obs:cumu}
If $\mu$ is a non-atomic measure on $\uint$, then {its cumulative distribution
function $\cdf_\mu(x) :=\mu([0,x])$} is a continuous increasing function which
has Luzin's (N) iff $\mu \ll \lambda$.
\end{observation}

In \cite{merkle3}, {Bienvenu and Merkle have done an extensive survey
of the conditions under which two computable measures $\mu$ and $\nu$
share the same randoms for a variety of notions of randomness
(Kurtz, computable, Schnorr, MLR, and weak-2-random).
Two trivial situations where $\mu$-randomness and $\lambda$-randomness
fail to coincide is if $\mu$ has an
atom or if $\mu(J)=0$ for some open interval $J$.
When discussing the connections among Luzin's (N), randomness reflection,
and coincidence of randomness notions, we will restrict our attention to computable
measures $\mu$ which avoid these two degenerate situations.  When $\mu$
is atomless, $\cdf_\mu$ is continuous and computable.  To say $\mu(J)>0$ for all
open intervals $J$, it is equivalent to say that
$\cdf_\mu$ is strictly increasing.  When the degenerate situations
are avoided,
$\cdf_\mu$ is a computable homeomorphism of $\uint$, so
$\cdf^{-1}_\mu$ is also a computable homeomorphism.  In this situation,
randomness reflection for $\cdf_\mu$ is exactly randomness
preservation for $\cdf^{-1}_\mu$.  }

\begin{proposition}\label{prop:reflection}
Let $\mu$ be a non-atomic computable probability measure on $\uint$ {with
$\cdf_\mu$ strictly increasing.} Then $x$ is $\mu$-MLR ($\mu$-Schnorr random, $\mu$-Kurtz random, $\mu$-$\Delta^1_1$-random) iff $\cdf_\mu(x)$ is Martin-L\"of random (Schnorr random, Kurtz random, $\Delta^1_1$-random) w.r.t.~the Lebesgue measure.
\begin{proof}
{For any set $A$, we have $\mu(A) = \lambda(\cdf_\mu(A))$, and $\cdf_\mu$
and $\cdf_\mu^{-1}$ are both computable homeomorphisms. }
We can thus move any relevant test from domain to codomain and vice versa.
\end{proof}
\end{proposition}

{Therefore, $\cdf_\mu$ reflects a given notion of
randomness exactly when the
$\mu$-randoms are contained in the $\lambda$-randoms for that notion of randomness.
Similarly, $\cdf_\mu^{-1}$ reflects a given notion of randomness exactly when
the $\lambda$-randoms are contained in the $\mu$-randoms.}

{Using our previous results, we obtain the following corollary.
The equivalence of (1) and (4) was proved in}
(\cite[Proposition 58]{merkle3}), but the others are new.
\begin{corollary}\label{corr:measure_equivalent}
{The following are equivalent for a computable probability measure $\mu$.}
\begin{enumerate}
\item $\mu$ is mutually absolutely continuous with the Lebesgue measure.
\item $\cdf_\mu$ is a homeomorphism and both $\cdf_\mu$ and $\cdf_\mu^{-1}$
have Luzin's (N).
\item $\mu$-$\Delta^1_1(\KO)$-randomness and $\Delta^1_1(\KO)$-randomness coincide.
\item $\mu$-weak-2-randomness and weak-2-randomness coincide.
\item $\mu$-Kurtz($\emptyset'$)-randomness and Kurtz($\emptyset'$)-randomness coincide.
\end{enumerate}
\end{corollary}
\begin{proof}
{First observe that in all cases above, $\cdf_\mu$ is a homeomorphism.  That is because
none of the cases is compatible with $\mu$ having an atom or assigning measure 0 to an interval.

 Then $(1) \iff (2)$ follows from Observation} \ref{obs:cumu} {for the case of $\cdf_\mu$,
 and by similar reasoning for the case of $\cdf_\mu^{-1}$.

 Since $\cdf_\mu$ and $\cdf_\mu^{-1}$ are computable functions of bounded variation,
 by Theorems }\ref{theo:main}, and \ref{theo:boundedvariation},
 {they have Luzin's (N) if and only if they
 reflect each kind of randomness mentioned in $(3)$-$(6)$.  So the implications
 $(2)\iff(3), (2)\iff(4),$ and $(2)\iff(5)$ now
 follow from Proposition} \ref{prop:reflection}.
\end{proof}

Bienvenu and Merkle also give some separations.
In particular, they show as \cite[Proposition 51 a)]{merkle3} that there exists a computable probability measure $\mu$ which is mutually absolutely continuous with Lebesgue measure, but $\mu$-MLR does not coincide with $\lambda$-MLR, $\mu$-Schnorr random does not coincide with with $\lambda$-Schnorr random, and $\mu$-computably random does not coincide with $\lambda$-computably random. Essentially, $\mu$ is obtained by thinning out the Lebesgue measure around Chaitin's $\Omega$ in a way that derandomizes $\Omega$ without introducing new null sets.

\begin{corollary}\label{cor:n_notimply_mlrr}
Luzin's (N) does not imply any of Martin-L\"of randomness reflection, Schnorr randomness reflection nor computable-randomness reflection; even for strictly increasing computable functions.
\end{corollary}
\begin{proof}
{If Luzin's (N) were to imply reflection for any of these kinds of randomness,
they could be included
in the list in Corollary} \ref{corr:measure_equivalent} {by the same reasoning,
but this would contradict
Bienvenu and Merkle's result above.}
\end{proof}

We still need to discuss reflection of (unrelativized) Kurtz randomness. In \cite[Proposition 56]{merkle3}, Bienvenu and Merkle construct a non-atomic computable probability measure $\mu$ such that $\mu$-Kurtz random and Kurtz random coincide, yet makes the Lebesgue measure not absolutely continuous relative to $\mu$. The construction is based on an involved characterization of $2$-randomness in terms of Kolmogorov complexity obtained by Nies, Stephan and Terwijn \cite{NST}. We could already conclude {that Kurtz randomness
reflection does not imply Luzin's (N)} from this, but instead we will provide a direct, more elementary construction in the following. Our separation works ``the other way around'', that is we obtain a probability measure $\mu$ which is not absolutely continuous w.r.t.~the Lebesgue measure. This shows that the Lebesgue measure has no extremal position for relative absolutely continuity inside the class of measures having the same Kurtz randoms. For comparison, a measure  satisfies Steinhaus theorem iff it is absolutely continuous w.r.t. Lebesgue measure \cite{mospan}.

\begin{theorem}
\label{theo:sepkurtz}
There is an increasing {surjective} computable function $f : \uint \to \uint$ which is not absolutely continuous, yet for any $\Pi^0_1$ set $A$ {with $\lambda(A) = 0$, }
it holds that $\lambda(f(A)) = 0$.
\end{theorem}

\begin{corollary}
There is a non-atomic probability measure $\mu$ such that $\mu$-Kurtz random and Kurtz random coincide, yet $\mu \not\ll \lambda$.
\end{corollary}
\begin{proof}
{Let $\hat \mu$ be the probability measure whose cumulative distribution function is $f$,
equivalently $\hat \mu(B) := \lambda(f(B))$.
Since $f$ does not have Luzin's (N), there is some set $B$ with $\lambda(B) = 0$
and $\hat\mu(B) >0$.  Let $\mu = \frac{1}{2}\hat\mu + \frac{1}{2}\lambda$.
Then using the same $B$, we see that $\mu \not\ll \lambda$.  On the other hand,
if $A$ is a $\Pi^0_1$ set, then $\lambda(A) = 0$ implies $\hat \mu(A) = 0$,
and thus $\lambda(A) = 0$ if and only if $\mu(A) = 0$.}
\end{proof}

\begin{corollary}
For increasing computable functions $f : \uint \to \uint$, reflecting Kurtz randomness
{does not imply} Luzin's (N).
\end{corollary}

We prepare our construction. Suppose $h:[0,1]\rightarrow [0,1]$ is a piecewise linear increasing function,
$B \subseteq [0,1]$ is a finite union of intervals {with rational endpoints}, and $\delta>0$.
We define a new function
$$\concentrate(h,B,\delta):[0,1]\rightarrow [0,1]$$
which concentrates $\lambda(B)$-much measure onto a set of Lebesgue
measure at most $\delta$, as follows.

\begin{definition}[Definition of Concentrate]
Given $h, B, \delta$ as above, write $B = \cup_{k<n} I_k$
where $I_k$ are almost disjoint intervals and $h\uhr h^{-1}(I_k)$ is
linear (contained in a single piece of the piecewise function).  Modify
$h$ on each interval $h^{-1}(I_k)$ by substituting a piecewise linear
function which alternates between a slope of 0 and a large positive slope.
The modification is chosen in a canonical computable way
to obtain the following outcomes.  Below, $\hat h$ denotes $\concentrate(h,B,\delta)$.
\begin{enumerate}
\item $h = \hat h$ outside of $h^{-1}(B)$.
\item Letting $F$ denote the union of the pieces of $\hat h^{-1}(B)$ which
have positive slope, we have $\lambda(F)<\delta$ and $f(F) = B$, and
\item For all $x$, $|h(x) - \hat h(x)|<\delta$.
\end{enumerate}
\end{definition}

\begin{lemma}
\label{lem:sepkurtz}
Suppose that $(B_n)_{n \in \mathbb{N}}$ is a computable {sequence} of finite unions of intervals in $\uint$.  Define a sequence of functions
$(f_n)_{n \in \mathbb{N}}$ inductively by setting $f_0(x) = x$ and
$$f_{n+1} = \concentrate(f_n, B_n, 2^{-n}).$$
Then $(f_n)_{n \in \mathbb{N}}$ converges uniformly to a computable increasing
function $f$.  Furthermore, if there is some $\varepsilon>0$ such that $\lambda(B_n)>\varepsilon$
for all $n$, then $f$ fails Lusin's (N).
\end{lemma}
\begin{proof}
The uniform convergence to a computable $f$ follows from the third property
in the definition of $\concentrate$, {and $f$ is increasing because each $f_n$
is}.  Observe that $\concentrate$ never
changes the value of $h$ at a break point of $h$.  Therefore,
the second property in the definition of $\concentrate$, which tells
us that $f_n(F) = B$ for some $F$ with $\lambda(F)<2^{-n}$,
implies that $f(F) = B$ as well (here we also used the fact that
$f$ is continuous and increasing).  It follows that $f$ is not absolutely
continuous, and thus fails Lusin's (N).
\end{proof}

\begin{proof}[Proof of Theorem \ref{theo:sepkurtz}]
We construct a computable sequence $(B_n)_{n \in \mathbb{N}}$
such that $\lambda(B_n)> 1/2$ for all $n$, and argue that
the function $f$ constructed as in Lemma \ref{lem:sepkurtz} satisfies $\lambda(f(P))=0$ whenever $P \in \Pi^0_1$
and $\lambda(P)=0$.

The strategy for a single $\Pi^0_1$ class $P_e$ is as follows.
Let $C_{e,0}$ be some interval of length $\varepsilon_e$.
Let $B_s = [0,1]\setminus C_{e,s}$.
As long as $f_s(P_{e,s})\cap C_{e,s}$ has measure at least $\varepsilon_e/2$,
define $C_{e,s+1} = C_{e,s}$.
If $f(P_{e,s})\cap C_{e,s}$ has measure less than
$\varepsilon_e/2$, define $C_{e,s+1} = (f_s(P_{e,s}) \cap C_{e,s}) \cup C$,
where $C$ is a new interval or finite union of intervals almost disjoint
from $\cup_{t\leq s} C_{e,t}$.  Choose $C$ so that that $C_{e,s+1}$
has measure $\varepsilon_e$, if possible; if this is not
possible, {choose $C$ so that} $\cup_{t\leq s+1} C_{e,t} = [0,1]$.  In
the latter case the measure of $C_{e,s+1}$ may be less than
$\varepsilon_e$ and this is
also fine.  If we reach this degenerate situation, we also stop checking
the measures and simply let
$C_{e,t} = C_{e,s+1}$ for all $t>s$.

We claim that if $\lambda(P_e) = 0$, then $\lambda(f(P_e)) = 0$.
Suppose at some stage $s$ we have that the measure of
$f_s(P_{e,s}) \cap C_{e,s}$ is greater than $\varepsilon_e/2$.
If this continues for all $t>s$, then $f$ and $f_s$ coincide on
the set $J:= f^{-1}(C_{e,s})$.  It follows that $f$ is piecewise linear
on $J$, but $f(P_e\cap J)$ has positive measure; this is impossible
since $P_e$ has measure 0.  We conclude that nothing lasts forever;
eventually we do reach a stage $s$ where
$\cup_{t\leq s} C_{e,s} = [0,1]$.  Since $C_{e,s}$ never changes again,
$f$ and $f_s$ again coincide on {$J:= f^{-1}(C_{e,s})$}.  Observe also that
$P_e \subseteq J$.  Since $f_s$ is piecewise linear and $\lambda(P_e) = 0$,
we also have $\lambda(f_s(P_e)) = 0$, and thus $\lambda(f(P_e)) = 0$.

The above strategy works purely with negative requirements, specifically
freezing $f$ on $f_s^{-1}(C_{e,s})$.  If other requirements also freeze
$f$ on other places, it has no effect on the proof above. The only thing
to consider when
combining requirements is that we need to make sure $\lambda(B_s)>1/2$
for all $s$, where we now define
$$B_s = [0,1] \setminus \bigcup_{e<s} C_{e,s}.$$
Since we always have $\lambda(C_{e,s}) \leq \varepsilon_e$,
we can keep the sets $B_s$ large by choosing the values of $\varepsilon_e$
to satisfy $\sum_e \varepsilon_e < 1/2$.
\end{proof}

\section{$\Pi^1_1$-hardness of randomness reflection}
\label{sec:generichardness}
If we do not restrict the domain of the functions to (locally) compact spaces, then essentially any form of randomness reflection is $\Pi^1_1$-hard. We show a construction which yields a function having either null range, or is surjective when restricted to a specific null subset of its domain. In particular, our construction is independent of the randomness notions involved.

\begin{theorem}
\label{theo:generichardness}
Let $K, L \subseteq \uint^2$ be non-empty sets containing only Kurtz randoms. Then ``whenever $f(x) \in K$, then already $x \in L$'' is a $\Pi^1_1$-hard property of continuous functions $f : (\uint \setminus \mathbb{Q}) \times \uint \to \uint^2$.
\begin{proof}
It is well-known that $\uint \setminus \mathbb{Q}$ and $\Baire$ are homeomorphic, and even computably so. We identify the spaces in such a way that the Lebesgue measure induced on $\Baire$ satisfies $\lambda(\{p \in \Baire \mid \forall n \ p_{2n} = p_{2n+1}\}) = 0$.

We construct a function $f_T : \Baire \times \uint \to \uint^2$ from a countably-branching tree $T$. First, we modify $T$ to obtain $\hat{T} = \{w_0w_0w_1w_1\ldots w_{n-1}w_{n-1}w_n \mid w \in T\} \cup \{w_0w_0w_1w_1\ldots w_{n-1}w_{n-1}w_nw_n \mid w \in T\}$. Clearly, $T$ is well-founded iff $\hat{T}$ is, and $[\hat{T}]$ contains no Kurz-randomns (so in particular,$[\hat{T}] \times \uint \cap L = \emptyset$). For any $p \in \Baire$, let $|T,p| = n$ iff $n$ is minimal such that $p_{\upharpoonright n} \notin \hat{T}$, and $|T,p| = \infty$ if $p \in [\hat{T}]$.

Let $s_\infty : \uint \to \uint^2$ be a computable space-filling curve, and let $(s_n)_{n \in \mathbb{N}}$ be a computable fast Cauchy sequence converging to $s_\infty$ such that any $s_n(\uint)$ is a finite union of line segments. We then define $f_T(p,x) = s_{|T,p|}(x)$. This construction is computable in $T$. We claim that $f_T$ has our reflection property iff $T$ is well-founded.

If $T$ is well-founded, then the range of $f_T$ is $\bigcup_{n \in \mathbb{N}} s_n(\uint)$. Since any $s_n(\uint)$ is a null $\Pi^0_1$-set, we see that $f_T$ never takes any Kurtz random values (in particular, none in $K$), and thus vacuously, if $f(x) \in K$ then $x \in L$. The argument in fact establishes that for arbitrary $T$, whenever $p \notin \hat{T}$ then $f_T(p,x)$ is not Kurtz random regardless of $x$.

Now assume that $T$ is ill-founded and that $y \in K$. We find that $f_T^{-1}(\{y\}) = [\hat{T}] \times s_\infty^{-1}(\{y\})$. Since $\hat{T}$ is illfounded and $s_\infty$ is space-filling, this set is non-empty. But by construction of $\hat{T}$, it cannot contain any elements of $L$. Hence, $f_T$ does not have our reflection property.
\end{proof}
\end{theorem}

\section{A glimpse at related notions}
\label{sec:digression}
As a slight digression, we have a look at related properties of functions, namely those where the image of null sets are required to belong to some other ideals of \emph{small} sets, such as being countable or being meager. These properties were investigated by Sierpinski \cite{siep34} and Erd\"os \cite{erdos43}, amongst others. Our results are formulated in some generality, but as a consequence, we do see that we do not get any ``regular" functions with these properties. In contrast, Erd\"os showed that under $\mathrm{CH}$ there is a bijection $f : \mathbb{R} \to \mathbb{R}$ mapping meager sets to null sets with $f^{-1}$ mapping null sets to meager sets.

\begin{theorem}
\begin{itemize}
\item[(1)] If $A$ is a nonnull $\mathbf{\Sigma}^1_1$ set and  $f$ is a continuous function mapping any null subset of $A$ to a countable set, then  the range of $f$ restricted to $A$  is countable. In particular, if $A$ is an interval, then  range of $f$ restricted to $A$  is  a constant function.
\item[(2)] Assume $\mathrm{CH}$, there is a function $f$ mapping any null set to a countable set  such that the range of $f$ is $\mathbb{R}$,  and $f(A)$ is uncountable for any nonnull set $A$ but for every $y$, $f^{-1}(y)$ is an uncountable Borel null set.
\item[(3)] If $A$ is a nonnull set and  $f$ is a continuous function mapping any null subset of $A$ to a meager set, then  the range of $f$ restricted to $A$  is meager. In particular, if $A$ is an interval, then  range of $f$ restricted to $A$  is  a constant function.
\item[(4)] If $f$ is a measurable function and maps a null set to a meager set, then the range of $f$ is meager. In particular, if $f$ is continuous with the property, then $f$ is constant.
\item[(5)] If $f$ has the Baire property and maps a meager set to a null set, then the range of $f$ is null. In particular, if $f$ is continuous with the property, then $f$ is constant.
\end{itemize}
\end{theorem}
\begin{proof}

(1). Fix a real $x$ so that $f$ is computable in $x$ and $A$ is $\Sigma^1_1(x)$.  Now for any real $z\in A$, let $g$ be a $\Delta^1_1( \KO^{x\oplus z})$-generic real.  Then $z$ cannot be Martin-L\" of random relative to $g$ and so there there must be a $\Delta^1_1(g)$-null set $G$ so that $z\in G\cap A$. By the assumption,  $f(G\cap A)$ is a  $\Sigma^1_1(x\oplus g)$-countable set and so every real in $f(G\cap A)$ is hyperarithmetically below $x\oplus g$. In particular, $f(z)\leq_h   x\oplus g$.  Since $f$ is computable in $x$, we also have that $f(z)\leq_h x\oplus z$. Then $f(z)\leq_h x$ since $g$ is $\Delta^1_1( \KO^{x\oplus z})$-generic real. By the arbitrarility of $z$, the range of $f$ restricted to $A$ is countable.  So if $A$ is an interval, then  range of $f$ restricted to $A$  is  a constant function.

(2). Fix an enumeration of nonempty $G_{\delta}$-null sets $\{G_{\alpha}\}_{\alpha<\aleph_1}$ and  all the reals $\{y_{\alpha}\}_{\alpha<\aleph_1}$. We define $f$ and $\{\beta_{\alpha}\}_{\alpha<\aleph_1}$ by induction on $\alpha$.

\bigskip

At stage $0$, define $f(x)=y_0$ for any $x\in G_0$. Define $\beta_0=0$.

At stage $\alpha<\aleph_1$, let $\beta_{\alpha}$ be the least ordinal $\gamma$ so that $G_{\gamma}\setminus \bigcup_{\alpha'<\alpha}(\bigcup_{\gamma'\leq \beta_{\alpha'} }G_{\gamma'})$ is uncountable. Define $f(x)=y_{\alpha}$ for any $x\in G_{\gamma}\setminus \bigcup_{\alpha'<\alpha}(\bigcup_{\gamma'\leq \beta_{\alpha'} }G_{\gamma'})$.

\bigskip

Clearly the range of $f$ is $\mathbb{R}$. Moreover, for any $\alpha<\aleph_1$, $f^{-1}(y_{\alpha})=G_{\beta_{\alpha}}\setminus \bigcup_{\alpha'<\alpha}(\bigcup_{\gamma'\leq \beta_{\alpha'} }G_{\gamma'})$ is an uncountable Borel null set.  Now for any null set $A$, there must be some $\alpha<\aleph_1$ so that $A\subseteq G_{\alpha}$. By the construction, $f(A)\subseteq f(G_{\alpha})\subseteq \{y_{\beta}\mid \beta\leq \alpha\}$ is a countable set.

(3). Fix a real $x$ so that $f$ is computable in $x$ restricted to $A$.  Fix a $2$-$x$-random real $r\in A$. Then $f(r)\leq_T x\oplus r$.  But $f(r)$ cannot be $2$-$x$-generic (see \cite{NST}). So the range  $f$ restricted to $A\cap\{r\mid r \mbox{ is  a }2-x\mbox{-random}\}$ is meager. But $A\setminus\{r\mid r \mbox{ is  a }2-x\mbox{-random}\}$ is a null set. So, by the assumption on $f$, the range of $f$ restricted to $A$ is meager.

(4). Suppose that $f$ is measurable function and maps a null set to a meager set. Without loss of generality, we may assume that the domain of $f$ is $[0,1]$. Then there is a sequence closed sets $\{F_n\}_{n\in \omega}$ so that $[0,1]\setminus \bigcup_{n\in \omega}F_n$ is null and $f$ restricted to $F_n$ is continuous for every $n$.   By (3), the range of $f$ restricted to $F_n$ is a meager set. So the range of $f$ restricted to $\bigcup_{n\in \omega} F_n$ is also a meager set. Note that $[0,1]\setminus \bigcup_{n\in \omega}F_n$ is null. So the range of $f$ is meager.  In particular, if $f$ is continuous with the property, then $f$ is constant.

(5). This is dual to (4).
\end{proof}

\section{Outlook}
\label{sec:outlook}
The most prominent avenue of future research seems to be the resolution of Question \ref{question:delta11}, asking for a separation (or equivalence proof) of $\Delta^1_1$-randomness reflection and $\Delta^1_1(\KO)$-randomness reflection. There are a few further aspects that merit further investigation, though.

\paragraph*{Topological properties}
While we have not been systematic in exploring the impact of topological properties of the domain (and maybe codomain) of the functions we explore, we observe that our proofs differ in the requirements they put on the spaces involved. {For example, the majority of the arguments presented in Sections }\ref{sec:mlr} and \ref{sec:fibers} {are relying just on the theory of randomness, and are thus applicable to any space where randomness works as usual (see }\cite{hoyrup5}).
In Section \ref{sec:kurtz}, (local) compactness of the domain is a core ingredient in our arguments.
In Section \ref{sec:boundedvariation} we do use particular properties of the reals, in particular connectedness. Further investigation of how topological properties of spaces relate to how randomness reflection behaves for functions on them seems warranted.

\paragraph*{Formalizing randomness reflection}
With the exception of Theorem \ref{theo:generichardness}, we have only considered symmetric notions of randomness reflection: Whenever $f(x)$ is random in some sense, we demand that $x$ is random in the very same sense. While this seems natural, a downside is that we do not get trivial implications between different notions of $\mathbf \Pi^0_2$-type randomness reflection.
We could consider the full \emph{square} of reflection notions, $(K,L)$-randomness reflection being that whenever $f(x)$ is $K$-random, then $x$ is $L$-random for randomness notions $K$,$L$. An extremal version also makes sense, where we just ask for when the image of all non-randoms under $f$ has positive measure. Whenever the latter property holds for some randomness notion $L$, then $f$ cannot have $(K,L)$-randomness reflection for any randomness notion $K$ at all. We typically prove non-randomness reflection in this manner.

It seems too early to pass judgement on what precise formulations of randomness reflection will ultimately be the most fruitful.

\paragraph*{Functions beyond measurability}
So far, the most general class of functions we considered for Luzin's (N) were the measurable functions. If we consider unrestricted functions in full generality, it is unsurprising that we quickly move beyond the confines of ZFC. For example, we are wondering whether Corollary \ref{corr:banach} holds for all functions having Luzin's (N)? An investigation into such questions is on its way by Yinhe Peng and the third author.

\bibliographystyle{eptcs}
\bibliography{references}

\section*{Acknowledgements}
This project was begun while all three authors were
in attendance at the Oberwolfach Seminar on
Computability Theory in January 2018.  The second
and third authors collaborated on it while attending
the IMS/NUS workshop Recursion Theory, Set Theory
and Interactions in May-June 2019.  The second author
was also partially
supported by the Cada R. and Susan Wynn Grove Early Career Professorship in Mathematics.  {The third author was partially supported by National Natural Science Fund of China, No. 11671196.}







\end{document}